\documentclass[12pt]{amsart}

\newcommand{\pdown}{\p_{\down}}
\newcommand{\pup}{\p^{\up}}

\newcommand{\side}{\te{side}}
\newcommand{\G}{\Gamma}
\newcommand{\simw}{\sim_{w}}

\newcommand{\tow}{\to w}

\newcommand{\lel}{\le_{L}}
\newcommand{\ler}{\le_{R}}

\newcommand{\lelr}{\le_{LR}}

\newcommand{\upl}{\ensuremath{\uplus}}

\newcommand{\M}{\ensuremath{\mathcal{M}}}

\newcommand{\eluw}{\ell(u, w)}

\newcommand{\subs}{\sub S}

\newcommand{\iW}{\ensuremath{\leftexp{I}W}}

\newcommand{\lw}{\ensuremath{\ell(w)}}

\newcommand{\ins}{\in S}
\newcommand{\indlw}{\in D_{L}(w)}

\newcommand{\elx}{\ensuremath{\ell(x)}}
\newcommand{\ely}{\ensuremath{\ell(y)}}

\newcommand{\dlw}{\ensuremath{D_{L}(w)}}
\newcommand{\drw}{\ensuremath{D_{R}(w)}}

\newcommand{\elw}{\ensuremath{\ell(w)}}
\newcommand{\elu}{\ensuremath{\ell(u)}}

\newcommand{\elv}{\ensuremath{\ell(v)}}

\newcommand{\down}{\downarrow}

\newcommand{\eluv}{\ensuremath{\ell(u, v)}}

\newcommand{\inw}{\ensuremath{\in W}}

\newcommand{\wh}{\ensuremath{\widehat}}
\newcommand{\el}{\ensuremath{\ell}}

\DeclareFontFamily{U}{mathx}{\hyphenchar\font45}
\DeclareFontShape{U}{mathx}{m}{n}{
      <5> <6> <7> <8> <9> <10>
      <10.95> <12> <14.4> <17.28> <20.74> <24.88>
      mathx10
      }{}
\DeclareSymbolFont{mathx}{U}{mathx}{m}{n}
\DeclareFontSubstitution{U}{mathx}{m}{n}
\DeclareMathAccent{\widecheck}{0}{mathx}{"71}

\newcommand{\tol}{\ensuremath{\to_{L}}}
\newcommand{\tor}{\ensuremath{\to_{R}}}

\newcommand{\tw}{\textwidth}

\renewcommand{\kill}[1]{}
\newcommand{\dummy}[1]{\mbox{}}

\makeatletter
\newcommand{\xequal}[2][]{\ext@arrow 0055{\equalfill@}{#1}{#2}}
\def\equalfill@{\arrowfill@\Relbar\Relbar\Relbar}
\makeatother

\newcommand{\mto}{\mapsto}

\newcommand{\Set}[2]{\ensuremath{\left\{{#1}\,\middle|\,{#2}\right\}}}
\newcommand{\ku}{\ensuremath{\emptyset}}

\renewcommand{\k}{\ensuremath{\ol{\mathrm{P}}}}

\newcommand{\vf}{\vfill}

\newcommand{\h}{\hline}

\newcommand{\ok}{\,\checkmark\,}

\renewcommand{\k}[1]{\ensuremath{\left({#1}\right)}}

\newcommand{\ds}{\dots}

\newcommand{\bca}{\begin{cases}}
\newcommand{\eca}{\end{cases}}

\newcommand{\A}{\mathcal{A}}

\newcommand{\Th}{\ensuremath{\Theta}}

\newcommand{\mug}{\ensuremath{\infty}}

\newcommand{\del}{\ensuremath{\delta}}

\newcommand{\ff}[2]{\ensuremath{\di\fr{#1}{#2}}}

\newcommand{\bpic}{\begin{picture}}\newcommand{\epic}{\end{picture}}

\newcommand{\beda}{\begin{edaenumerate}}
\newcommand{\eeda}{\end{edaenumerate}}

\newcommand{\cd}{\cdots}

\newcommand{\xto}{\xrightarrow}

\newcommand{\st}{\strut}

\newcommand{\qu}{\quad}
\newcommand{\q}{\quad}

\newcommand{\up}{\uparrow}

\newcommand{\push}[1]{\resizebox{0.9\hsize}{!}{#1}}

\newcommand{\bq}{\begin{quote}}\newcommand{\eq}{\end{quote}}

\newcommand{\ti}{\times}

\newcommand{\be}{\begin{enumerate}}\newcommand{\ee}{\end{enumerate}}
\newcommand{\bce}{\begin{center}}\newcommand{\ece}{\end{center}}
\newcommand{\bde}{\begin{description}}\newcommand{\ede}{\end{description}}
\newcommand{\bri}{\begin{flushright}}\newcommand{\eri}{\end{flushright}}
\newcommand{\bb}{\begin{block}}\newcommand{\eb}{\end{block}}
\newcommand{\bt}{\begin{thm}}\newcommand{\et}{\end{thm}}
\newcommand{\bpf}{\begin{proof}}\newcommand{\epf}{\end{proof}}
\newcommand{\bex}{\begin{ex}}\newcommand{\eex}{\end{ex}}
\newcommand{\bexr}{\begin{exr}}\newcommand{\eexr}{\end{exr}}
\newcommand{\bft}{\begin{fact}}\newcommand{\eft}{\end{fact}}
\newcommand{\brk}{\begin{rmk}}\newcommand{\erk}{\end{rmk}}
\newcommand{\ba}{\begin{align*}}\newcommand{\ea}{\end{align*}}
\newcommand{\bexe}{\begin{exe}}\newcommand{\eexe}{\end{exe}}
\newcommand{\tn}{\textnormal}

\newcommand{\bit}{\begin{itemize}}\newcommand{\eit}{\end{itemize}}

\newcommand{\bcm}{}

\newcommand{\hf}{\hfill}
\newcommand{\ci}{\CIRCLE}
\newcommand{\fr}{\frac}

\newcommand{\bd}{\begin{defn}}\newcommand{\ed}{\end{defn}}
\newcommand{\bp}{\begin{prop}}\newcommand{\ep}{\end{prop}}
\newcommand{\p}{\ensuremath{\pi}}
\newcommand{\eh}{\emph}
\newcommand{\sub}{\subseteq}
\newcommand{\lam}{\lambda}

\newcommand{\mb}{\mbox}
\newcommand{\te}{\text}
\newcommand{\wt}{\widetilde}\newcommand{\sm}{\setminus}

\newcommand{\then}{\Longrightarrow}
\newcommand{\leftexp}[2]{{\vphantom{#2}}^{#1}{#2}}
\newcommand{\di}{\displaystyle}

\renewcommand{\d}{\ensuremath{\bm{d}}}

\newcommand{\f}{\frac}

\newcommand{\np}{\newpage}

\newcommand{\av}{\tn{av}}

\renewcommand{\up}{\uparrow}

\renewcommand{\int}{\in T}

\renewcommand{\d}{\delta}
\renewcommand{\P}{\mathcal{P}}
\newcommand{\Q}{\mathcal{Q}}

\renewcommand{\d}{\delta}

\usepackage[dvipdfmx]{graphicx}
\usepackage[dvipsnames]{xcolor}
\usepackage{float,afterpage}
\usepackage{asymptote,layout}
\usepackage{wrapfig,epic}

\usepackage{geometry}
\usepackage{exscale,latexsym,bm}
\usepackage{amssymb,enumerate,amsmath,amsthm,amsfonts}
\usepackage{verbatim,fancybox,wasysym,fancyhdr,type1cm}
\usepackage[frame,all,poly,curve,knot,arrow]{xy}
\usepackage{colortbl}
\usepackage{boxedminipage}
\usepackage{multirow}

\graphicspath{{./figs/}}
\theoremstyle{definition}
\newtheorem{thm}{Theorem}[section]
\newtheorem{lem}[thm]{Lemma}
\newtheorem{prop}[thm]{Proposition}\newtheorem{cor}[thm]{Corollary}

\newtheorem{exr}[thm]{Exercise}
\newtheorem{ob}[thm]{Observation}

\newtheorem{ex}[thm]{Example}

\newtheorem{ques}[thm]{Question}

\newtheorem{defn}[thm]{Definition}\newtheorem{rmk}[thm]{Remark}
\newtheorem{fact}[thm]{Fact}
\newtheorem{block}[thm]{}
\newtheorem*{exe}{Exercise}

\newcommand{\out}{\tn{out}}
\newcommand{\minx}{\min X}

\newcommand{\inn}{\tn{in}}

\newcommand{\maxx}{\max X}

\newcommand{\inx}{\in X}

\newcommand{\col}{\tn{col}}
\renewcommand{\d}{\Delta}

\renewcommand{\d}{\delta}

\title[]{Construction of double coset system of a Coxeter group 
and its applications to Bruhat graphs}
\author[Masato Kobayashi]{Masato Kobayashi$^{*}$}
\date{\today}
\address{Department of Engineering\\
Kanagawa University, 3-27-1 Rokkaku-bashi, Yokohama 221-8686, Japan.}
\thanks{*Department of Engineering, Kanagawa University, Japan}
\email{masato210@gmail.com}
\newcommand{\self}{\tn{self}}
\newcommand{\md}{\tn{mid}}

\newcommand{\pd}{\pi_{\down}}
\renewcommand{\P}{\ensuremath{\mathcal{P}}}
\renewcommand{\Q}{\ensuremath{\mathcal{Q}}}

\renewcommand{\G}{\ensuremath{\mathcal{G}}}

\renewcommand{\G}{\Gamma}
\newcommand{\D}{\Delta}

\begin{document}
\begin{abstract} 
We develop combinatorics of parabolic double cosets in finite Coxeter groups as a follow-up of recent articles by 
Billey-Konvalinka-Petersen-Slofstra-Tenner and Petersen. 
(1) We construct a double coset system as a generalization of a two-sided analogue of a Coxeter complex and present its order structure with  its local dimension function on certain connected components.
As applications of double cosets to Bruhat graphs, we also prove:
(2) every parabolic double coset is regular, 
(3) invariance of degree on Bruhat graph on lower intervals as an analogy of the one for Kazhdan-Lusztig polynomials, 
(4) every noncritical Bruhat interval satisfies out-Eulerian property.
\end{abstract}
\maketitle
\tableofcontents

\newcommand{\dimdw}{\dim_{\D(w)}}
\newcommand{\xiw}{\Xi(w)}
\newcommand{\Xiw}{\Xi(w)}

\section{Introduction}
\newcommand{\ixj}{(I, X, J)}


\subsection{double cosets}

Every Coxeter system possesses graded poset structures 
simultaneously as a (left/right) weak order and Bruhat order. 
These orders often show up in many topics:
Coxeter complex, hyperplane arrangement, Eulerian polynomials, cluster algebras, Kazhdan-Lusztig polynomials and so on. 
When we consider two weak orders together (two-sided order),  interesting interactions come into play which we must carefully analyze.  One such example is to investigate structures of ``double cosets" such as Solomon algebra or contingency tables. Recently, there are new  developments in this topic:
\begin{itemize}
	\item a two-sided analogue of a Coxeter complex and Eulerian polynomials: Petersen \cite{petersen1,petersen2} in 2018 and 2013.
	\item enumeration of double cosets with its minimal representative  fixed: Billey-Konvalinka-Petersen-Slofstra-Tenner \cite{billey5} in 2018.
\end{itemize}

%
%

{\renewcommand{\arraystretch}{1.75}
\begin{table}[h]
\caption{Three systems on $W$}
\label{t1}
\begin{center}
\resizebox{\tw}{!}{
	\begin{tabular}{|c|c|c|ccccc}\h
	one-sided Coxeter system&two-sided analogue  	&double coset system	\\\h
	$\Sigma(W)$ &$\Xi(W)$	&$\D(W)$	\\\h
	one-sided cosets&marked double cosets &	double cosets\\\h
	Tits&Hultman, Petersen	&not studied	\\\h
	simplicial complex&boolean complex	&	?\\\h
	$\dim(xW_{I})=|S\sm I|-1$&$\dim(I, X, J)=|S\sm I|+|S\sm J|-1$	&?	\\\h
	max $\dim= n-1$&max $\dim= 2n-1$&?\\\h
\end{tabular}
}
\end{center}
\end{table}}

\subsection{main results}
The aim of this article is to study combinatorics of parabolic double cosets and Bruhat graphs in finite Coxeter groups as a follow-up of their papers. We prove four main results as theorems:

\begin{enumerate}
	\item Theorem \ref{th1}: we construct a double coset system as a two-sided analogue of Coxeter complex and present several its order  structures together with 
one-sided Coxeter complex and Petersen's analogue (Table \ref{t1}). 
	\item Theorem \ref{th2}: every parabolic double coset is regular, 
	\item Theorem \ref{th3}: invariance of degree on Bruhat graph as an analogy of the one for Kazhdan-Lusztig polynomials, 
	\item Theorem \ref{th4}: every noncritical Bruhat interval satisfies out-Eulerian property.
\end{enumerate}

\subsection{Organization}

Section 2 gives basic ideas and facts on Coxeter systems. 
In Section 3, we construct a double coset system and show several results on its  order structures.
In Section 4, as applications of the idea double cosets (as Bruhat intervals), we prove three theorems on degree of a vertex on Bruhat graphs. In Section 5, we record some ideas and open problems for our research in the future.

\section{Preliminaries on Coxeter groups}

Throughout this article, we denote by $W=(W, S, T, \el,\le)$ a Coxeter system with $W$ the underlying Coxeter group, $S$ its Coxeter generators, $T$ the set of its reflections, $\el$ the length function, $\le$ Bruhat order. Moreover, assume that $W$ is finite of rank $n=|S|$.
Unless otherwise noticed, symbols $u, v, w, x, y$ are elements of $W$, $r, s\in S$, $t\in T$, $e$ is the group-theoretic unit of $W$ and $I, J$ are subsets of $S$.
The symbol $\el(u, v)$ means $\el(v)-\elu$ for $u\le v$, 
$x\lhd y$ a cover relation in a poset, and 
$w=u*v$ a reduced factorization of $w$:
$w=uv$ as a group element, and $\elw=\elu+\elv$.

\subsection{weak, two-sided, Bruhat orders}

We begin with basic definitions on partial orders on $W$.

\begin{defn}Write
\begin{enumerate}
	\item $u\lhd_{L} v$ if $\eluv=1$ and $v=su$ for some $s\ins$.
	\item $u\lhd_{R} v$ if $\eluv=1$ and $v=us$ for some $s\ins$.
	\item $u\lhd_{LR} v$ if $u\lhd_{L}v$ or $u\lhd_{R} v$. 
	\item $u\lhd v$ if $\eluv=1$ and $v=tu$ for some $t\in T$
	(equivalently, $v=ut'$ for some $t'\in T$).
\end{enumerate}
(Further, we sometimes write 
$u\lhd_{2LR}v$ to mean $u\lhd_{L}v$ and $u\lhd_{R}v$.) 
Define four partial orders on $W$, 
\eh{left weak order} $(\lel)$, 
\eh{right weak order} $(\ler)$, 
\eh{two-sided order} $(\lelr)$ and 
\eh{Bruhat order} $(\le)$, as transitive closure of those four binary relations on $W$. 
The interval notation is, for example, 
\[
[u, w]_{LR}=\{v\in W\mid u\lelr v\lelr w\}.
\]
\end{defn}

Say $s$ is a \eh{left (right) descent} of $w$ if 
$\el(sw)<\elw$ ($\el(ws)<\elw$).
Say $s$ is a \eh{left (right) ascent} of $w$ if 
$\el(sw)>\elw$ ($\el(ws)>\elw$). 
We denote these sets of descents and ascents as follows:
\begin{align*}
	D_{L}(w)&=\{s\in S\mid \el(sw)<\elw\},
	\\D_{R}(w)&=\{s\in S\mid \el(ws)<\elw\},
	\\A_{L}(w)&=\{s\in S\mid \el(sw)>\elw\}\,\,(=S\sm \dlw),
	\\A_{R}(w)&=\{s\in S\mid \el(ws)>\elw\}\,\,(=S\sm \drw).
\end{align*}
Also, the set of \eh{left, right inversions} are 
\[T_{L}(w)=\{t\in T\mid \el(tw)<\elw\},\]
\[T_{R}(w)=\{t\in T\mid \el(wt)<\elw\}.\]

\subsection{Bruhat graphs}

\bd{The \emph{Bruhat graph} of $W$ is a directed graph for vertices $w\in W$ and for edges $u\to v$. 
For each subset $V\sub W$, we can also consider the induced subgraph with the vertex set $V$ (Bruhat subgraph).
}\ed 



For convenience, we say an edge $x\to y$ is \eh{short} if $\ell(x, y)=1$ 
(i.e., $x\lhd y$). 
In other words, the short Bruhat graph for $V$ is the directed version of the Hasse diagram of $V$.

Each subset of $W$ can be regarded as a subposet under several kinds of graded partial orders (left weak, right weak, two-sided or Bruhat order). 


\subsection{double cosets}

Most of our results rely on the following important property of Bruaht order:
\begin{fact}
\label{lifp}
Let $u<w$. If $s\in A_L(u)\cap D_L(w)$, then 
$su\le w$ and $u\le sw$ (Lifting Property).
Consequently, if $v\in [u, w]$ and $s\in A_{L}(u)\cap D_{L}(w)$, 
then $sv\in [u, w]$.
\end{fact}

The right version of this property also holds.

\bd{Let $I\subseteq S$. By $W_I$ we mean the (standard) 
\eh{parabolic subgroup} of $W$ generated by $I$. 
A subset $X$ in $W$ is a \emph{parabolic double coset} if
 \[X=W_IxW_J=\{uxv \mid u\in W_I, v\in W_J\}\] 
for some $x\in W$ and $I, J\subseteq S$.
}\ed
In particular, every singleton set is a parabolic double coset itself 
while the empty set is not. By simply a coset or double coset, we mean a parabolic double coset hereafter.
\begin{fact}
\label{rep}Each double coset $X$ has the representative of maximal and minimal length:
there exists a unique pair $(x_{0}, x_{1})\in X\ti X$ such that 
\[
	\el(x_{0})\le \elx\le \el (x_{1})
\]
for all $x\in X$.
\end{fact}


\begin{ob}
A double coset $X$ is (by construction) an interval in LR order.
It is thus meaningful to write 
\[
X=``[x_{0}, x_{1}]_{LR}"
=\{x\in W\mid x_{0}\le_{LR} x\le_{LR} x_{1}\},\]
with $x_{0}=\minx, x_{1}=\maxx$.
Observe that if $X=
W_{I}xW_{J}=[x_{0}, x_{1}]_{LR}$, then 
$I\sub A_{L}(x_{0})\cap D_{L}(x_{1})$ 
and $J\sub A_{R}(x_{0})\cap D_{R}(x_{1})$ otherwise $x_{0}, x_{1}$ cannot be the extremal elements of 
$X$. 
The \eh{length} of a double coset $X$ is $\el(X)= \el(\min X, \max X)$.
Each $x\inx$ can be written as  
$x=u*x_{0}*v$ for \eh{some} $u\in W_{I}$, 
$v\in W_{J}$, $I, J\subs$ with 
$I\sub A_{L}(x_{0})\cap D_{L}(x_{1})$ and 
$J\sub A_{R}(x_{0})\cap D_{R}(x_{1})$. 
Call $u$ \eh{a} left part of $x$, 
$v$ \eh{a} right part of $x$,
and $x_{0}$ \eh{the} central part of $x$ (in $X$).
\end{ob}



\subsection{presentations of a double coset}



We just introduced a double coset as a set in the form $X=W_IxW_J$. 
If $J=\varnothing$ at the extreme case, then $X=W_IxW_J=W_Ix$ is an ordinary left coset.
However, it is worth mentioning that the double coset $W_IxW_J$ may be equal to $W_Ix$ (as sets) even if $J\ne\varnothing$ (or to $xW_J$ even if $I\ne\varnothing$). 
For example, the whole $W$ is itself a double coset $W_Sw_0W_S$ and furthermore \[
W=W_Sw_0W_S=W_{I}eW_S=W_SxW_{J}.
\] Hence there are many ways to express a double coset with a certain choice of $x$ and $I, J$. 

\begin{defn}
Let $X$ be a double coset. Say a triple $(I, x, J)$ is a \eh{presentation} of $X$ 
if $X= W_{I}xW_{J}$.
\end{defn}




\begin{prop}
Say a presentation $(I, x, J)$ of $X$
is \eh{maximal} if whenever 
\[
X=W_{I'}x'W_{J'},
\]
then $I'\sub I, J' \sub J$.
\end{prop}

Billey-Konvalinka-Petersen-Slofstra-Tenner \cite[Proposition 3.7]{billey5}
proved that there is a unique maximal presentation for each double 
coset:
\begin{fact}
Let 
\[
x_{0}=\minx, x_{1}=\maxx,
\]
\[
M_{L}(X)=A_{L}(x_{0})\cap D_{L}(x_{1}),
\]
\[
M_{R}(X)=A_{R}(x_{0})\cap D_{R}(x_{1}).
\]
Then, $(M_{L}(X), x_{1}, M_{R}(X))$ is a unique maximal presentation of $X$. 
\end{fact}
Hence this is \eh{the} maximal presentation of $X$.
Similarly, we can talk about \eh{a} minimal presentation of $X$ in the following sense:
a presentation $(I, x, J)$ of $X$ is \eh{minimal} if 
whenever $(I', x', J')$ is a presentation of $X$ and $I'\sub I, J'\sub J$, 
then $I'=I, J'=J$. However, for a given $X$, there may exist more than one minimal presentation.

%
%


\section{Double coset system}

\subsection{Petersen's two-sided analogue of the Coxeter complex}
We first review Petersen's two-sided analogue of the Coxeter complex 
of $W$ \cite{petersen1}.
Motivated by Hultman \cite{hultman}, 
he constructed it as a collection of 
\eh{marked double cosets}:
\[
\Xi=\Xi(W)=\{(I, W_{I}xW_{J}, J)\mid x\in W, I, J\subs\}.
\]
\renewcommand{\sup}{\supseteq}
He then introduced a partial order 
$(I, X, J)\le_{\Xi} (I, X', J')$ by $I\sup I'$, $J\sup J'$ and $X\sup X'$.
Each face (element) $F=(I, X, J)$ is colored by 
\[
\col(F)=(S\sm I, S\sm J)
\]
so that $\dim (F)=|S\sm I|+|S\sm J|-1$.

\begin{fact}[{\cite[Theorem 3]{petersen1}}]
For any Coxeter system $(W, S)$ with $|S|=n<\mug$, we have the following.
\begin{enumerate}
	\item The complex $\Xi$ is a balanced boolean complex of dimension $2n-1$.
	\item The facets (maximal faces) of $\Xi$ are in bijection with the elements of $W$, and the Coxeter complex $\Sigma$ is a relative subcomplex of $\Xi$.
	\item The complex $\Xi$ is shellable and any linear extension of the two-sided weak order on $W$ gives a shelling order for $\Xi.$
	\item If $W$ is finite then $\Xi$ is contractible.
	\item If $W$ is infinite, 
	\begin{enumerate}
	\item[(a)] the geometric realization of $\Xi$ is a sphere, and 
	\item[(b)] a refined $h$-polynomial of $\Xi$ is the 
	\eh{two-sided $W$-Eulerian polynomial}, 
	\[
h(\Xi, s,t)=\sum_{w\inw} s^{\tn{des}_{L}(w)}t^{\tn{des}_{R}(w)}	\]
where $\tn{des}_{L}(w)$ denotes the number of left descents of $w$ and $\tn{des}_{R}(w)$ denotes the number of right descents of $w$. 
\end{enumerate}
\end{enumerate}
\end{fact}


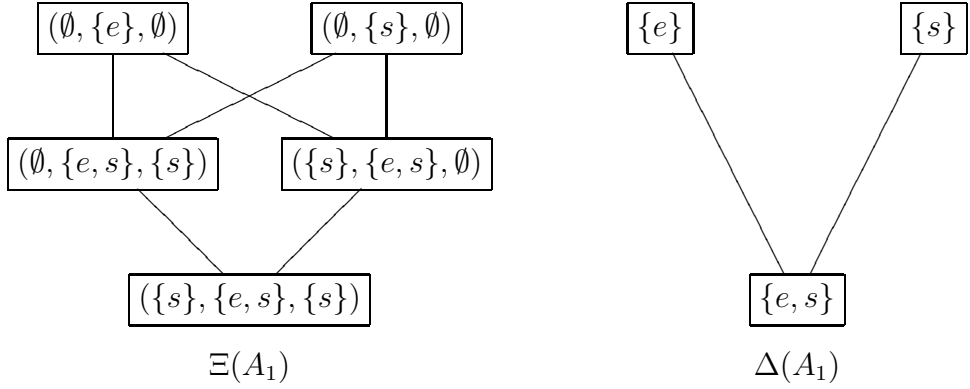
\begin{figure}
\caption{poset structures of $\Xi(A_{1})$ and $\D(A_{1})$
}
\label{f1}
\mb{}
\begin{center}
\begin{minipage}[t]{.\tw}
\begin{xy}
(0,0);<18mm,0mm>:
,(-1,2)*+[F]{(\ku,\{e\},\ku)}="11"
,(1,2)*+[F]{(\ku,\{s\},\ku)}="12"
,(-1,1)*+[F]{(\ku, \{e, s\}, \{s\})}="2"
,(1,1)*+[F]{(\{s\}, \{e, s\}, \ku)}="3"
,(0,0)*+[F]{(\{s\}, \{e, s\}, \{s\})}="4"
,(0,-0.5)*{\Xi(A_{1})}
,(4,-0.5)*{\D(A_{1})}
,(3,2)*+[F]{\{e\}}="a1"
,(4,0)*+[F]{\{e, s\}}="a0"
,(5,2)*+[F]{\{s\}}="a2"
,\ar@{-}"11";"2"
,\ar@{-}"11";"3"
,\ar@{-}"12";"2"
,\ar@{-}"12";"3"
,\ar@{-}"2";"4"
,\ar@{-}"3";"4"
,\ar@{-}"a0";"a2"
,\ar@{-}"a0";"a1"
\end{xy}
\end{minipage}
\end{center}
\end{figure}

%


\subsection{double coset system}

\begin{defn}
Define $\D=\D(W)$ be the set of all double cosets of $W$.
Introduce a partial order on $\D(W)$ by the reverse of containment:
\[
Y\le_{\D} X\iff X\sub Y.
\]
\end{defn}
Unlike one-sided and two-sided Coxeter complexes, $(\D, \le)$ is not necessarily a complex. However, it possesses some combinatorial structure with the ``local dimension" function as we will see details below.
For this reason, let us call the poset $(\D(W), \le)$ the \eh{double coset system} of $W$.

\begin{ex}\hf
\begin{enumerate}
	\item $W=A_{1}$ contains 5 marked cosets and 3 cosets (Figure \ref{f1}).
	\item $W=A_{2}$ contains 33 marked cosets
and 19 cosets (Figures \ref{f2} and \ref{f3}).
\end{enumerate}
\end{ex}




\begin{figure}
\caption{the boolean complex $\Xi(A_{2})$ (a copy of \cite[Figure 1]{petersen1})
}
\label{f2}
\mb{}
\begin{center}
\newcommand{\scb}[1]{\scalebox{5}{${#1}$}}
\renewcommand{\push}[1]{\resizebox{1\tw}{!}{#1}}
\push{
\xymatrix@R=200mm@C=10mm{
&
\scalebox{5}{$
(\ku, e, \ku)$}\ar@{-}[dl]\ar@{-}[d]\ar@{-}[dr]\ar@{-}[drr]&&
\scb{(\ku, s_{1}, \ku)}\ar@{-}[dll]\ar@{-}[dl]\ar@{-}[dr]\ar@{-}[drr]&&
\scb{(\ku, s_{2}, \ku)}\ar@{-}[dlllll]\ar@{-}[dll]\ar@{-}[drr]\ar@{-}[drrr]&&
\scb{(\ku, s_{1}s_{2}, \ku)}\ar@{-}[dlll]\ar@{-}[dr]\ar@{-}[drr]\ar@{-}[drrr]&&
\scb{(\ku, s_{2}s_{1}, \ku)}\ar@{-}[dllll]\ar@{-}[dll]\ar@{-}[drr]\ar@{-}[drrr]&&
\scb{(\ku, s_{1}s_{2}s_{1}, \ku)}\ar@{-}[dll]\ar@{-}[dl]\ar@{-}[d]\ar@{-}[dr]&\\
\scb{(\ku, e, 2)}
\ar@{-}[dr]\ar@{-}[drr]\ar@{-}[drrrr]
&
\scb{(\ku, e, 1)}\ar@{-}[d]\ar@{-}[drr]\ar@{-}[drrrr]&
\scb{(1, e, \ku)}\ar@{-}[d]\ar@{-}[dr]\ar@{-}[drrrrr]&
\scb{(2, e, \ku)}\ar@{-}[dr]\ar@{-}[drr]\ar@{-}[drrrr]&
\scb{(\ku, s_{1}, 2)}\ar@{-}[dlll]\ar@{-}[dll]\ar@{-}[drrrr]&
\scb{(2, s_{1}, \ku)}\ar@{-}[d]\ar@{-}[drr]\ar@{-}[drrr]&&
\scb{(\ku, s_{2}, 1)}\ar@{-}[dllllll]\ar@{-}[dll]\ar@{-}[drr]&
\scb{(1, s_{2}, \ku)}\ar@{-}[dllllll]\ar@{-}[dl]\ar@{-}[dr]&
\scb{(\ku, s_{1}s_{2}, 1)}\ar@{-}[dllllllll]\ar@{-}[d]\ar@{-}[dr]&
\scb{(2, s_{1}s_{2}, \ku)}\ar@{-}[dlll]\ar@{-}[dll]\ar@{-}[d]&
\scb{(\ku, s_{2}s_{1}, 2)}\ar@{-}[dllllllllll]\ar@{-}[dlll]\ar@{-}[d]&
\scb{(1, s_{2}s_{1}, \ku)}\ar@{-}[dlllll]\ar@{-}[dlll]\ar@{-}[dl]\\
&
\scb{(\ku, e, 12)}&\scb{(1, e, 2)}&\scb{(1, e, 1)}&\scb{(2, e, 2)}&\scb{(2, e, 1)}&&
\scb{(12, e, \ku)}&
\scb{(2,s_{1},2)}&
\scb{(1,s_{2},1)}&
\scb{(2, s_{1}s_{2}, 1)}&
\scb{(1, s_{2}s_{1},2)}&
\\
&&\scb{(1, e, 12)}\ar@{-}[ul]\ar@{-}[u]\ar@{-}[ur]\ar@{-}[urrrrrrr]\ar@{-}[urrrrrrrrr]
&&\scb{(2, e, 12)}\ar@{-}[ulll]\ar@{-}[u]\ar@{-}[ur]\ar@{-}[urrrrrr]\ar@{-}[urrrr]&&&&
\scb{(12, e, 1)}\ar@{-}[ulllll]\ar@{-}[ulll]\ar@{-}[ul]\ar@{-}[ur]\ar@{-}[urr]&&\scb{(12, e, 2)}\ar@{-}[ullllllll]\ar@{-}[ullllll]\ar@{-}[ur]\ar@{-}[ulll]\ar@{-}[ull]&&\\
&&&&&&\scb{(12, e, 12)}\ar@{-}[ull]\ar@{-}[ullll]\ar@{-}[urr]\ar@{-}[urrrr]&&&&&&\\
}
}
\end{center}
\end{figure}
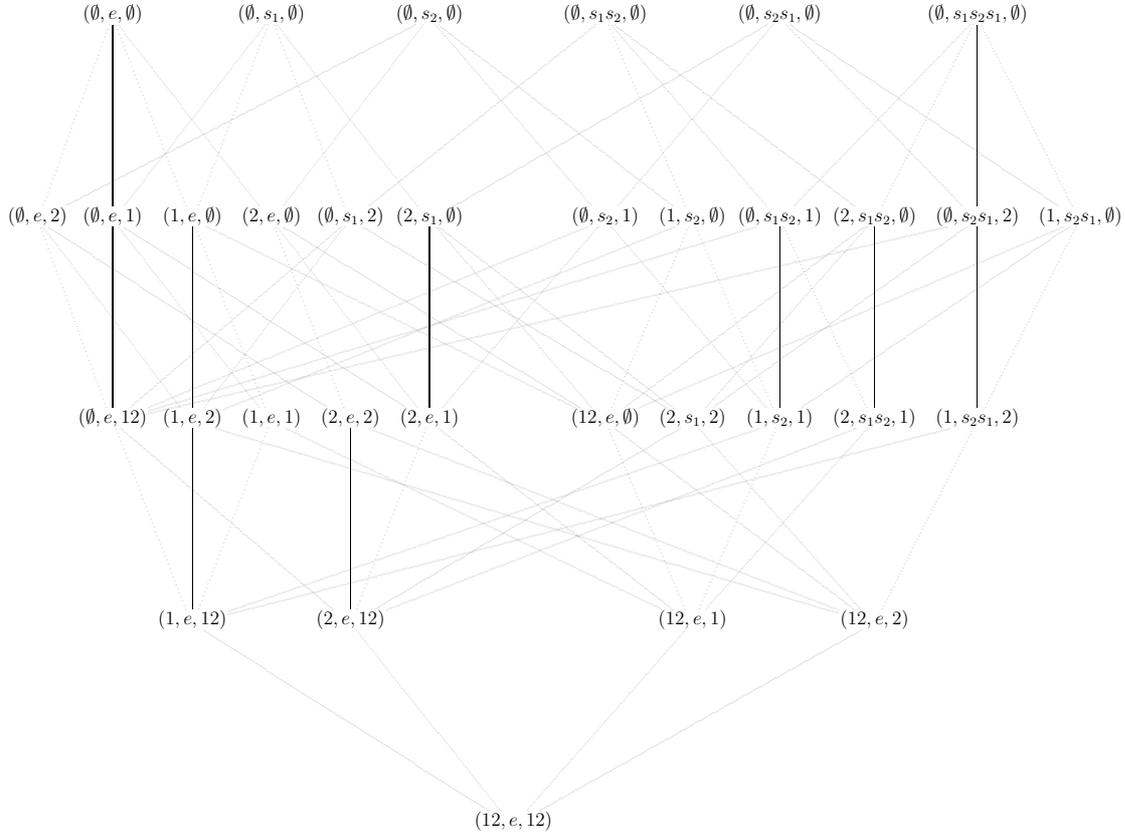

\begin{figure}
\caption{double coset system $\D(A_{2})$: 19 cosets}
\label{f3}
\begin{center}
\resizebox{1\tw}{!}{
\xymatrix@R=8mm@C=-15mm{
{\begin{xy}
0;<2.75mm,0mm>:
,0*{\ci}="b"*++++!U{}
,(-3,2)*{\scalebox{1.5}{$\circ$}}="al"*++++!R{}
,(3,2)*{\scalebox{1.5}{$\circ$}}="ar"*++++!L{}
,(-3,5)*{\scalebox{1.5}{$\circ$}}="cl"*++++!R{}
,(3,5)*{\scalebox{1.5}{$\circ$}}="cr"*++++!L{}
,(0,7)*{\scalebox{1.5}{$\circ$}}="t"*++++!D{}
,\ar@{--}"b";"ar"
,\ar@{--}"b";"al"
,\ar@{--}"al";"cl"
,\ar@{--}"ar";"cr"
,\ar@{--}"cl";"t"
,\ar@{--}"cr";"t"
,\ar@{--}"al";"cr"
,\ar@{--}"ar";"cl"
\end{xy}}\ar@{-}@/_10ex/[dddd]\ar@{-}[ddddrr]
&
\begin{xy}
0;<2.75mm,0mm>:
,0*{\scalebox{1.5}{$\circ$}}="b"*++++!U{}
,(-3,2)*{\ci}="al"*++++!R{}
,(3,2)*{\scalebox{1.5}{$\circ$}}="ar"*++++!L{}
,(-3,5)*{\scalebox{1.5}{$\circ$}}="cl"*++++!R{}
,(3,5)*{\scalebox{1.5}{$\circ$}}="cr"*++++!L{}
,(0,7)*{\scalebox{1.5}{$\circ$}}="t"*++++!D{}
,\ar@{--}"b";"ar"
,\ar@{--}"b";"al"
,\ar@{--}"al";"cl"
,\ar@{--}"ar";"cr"
,\ar@{--}"cl";"t"
,\ar@{--}"cr";"t"
,\ar@{--}"al";"cr"
,\ar@{--}"ar";"cl"
\end{xy}\ar@{-}[ddrrr]\ar@{-}[ddl]\ar@{-}[ddddl]
&
\begin{xy}
0;<2.75mm,0mm>:
,0*{\scalebox{1.5}{$\circ$}}="b"*++++!U{}
,(-3,2)*{\scalebox{1.5}{$\circ$}}="al"*++++!R{}
,(3,2)*{\ci}="ar"*++++!L{}
,(-3,5)*{\scalebox{1.5}{$\circ$}}="cl"*++++!R{}
,(3,5)*{\scalebox{1.5}{$\circ$}}="cr"*++++!L{}
,(0,7)*{\scalebox{1.5}{$\circ$}}="t"*++++!D{}
,\ar@{--}"b";"ar"
,\ar@{--}"b";"al"
,\ar@{--}"al";"cl"
,\ar@{--}"ar";"cr"
,\ar@{--}"cl";"t"
,\ar@{--}"cr";"t"
,\ar@{--}"al";"cr"
,\ar@{--}"ar";"cl"
\end{xy}\ar@{-}[dd]\ar@{-}[ddrrrr]\ar@{-}@/_10ex/[dddd]&{}&
\begin{xy}
0;<2.75mm,0mm>:
,0*{\scalebox{1.5}{$\circ$}}="b"*++++!U{}
,(-3,2)*{\scalebox{1.5}{$\circ$}}="al"*++++!R{}
,(3,2)*{\scalebox{1.5}{$\circ$}}="ar"*++++!L{}
,(-3,5)*{\ci}="cl"*++++!R{}
,(3,5)*{\scalebox{1.5}{$\circ$}}="cr"*++++!L{}
,(0,7)*{\scalebox{1.5}{$\circ$}}="t"*++++!D{}
,\ar@{--}"b";"ar"
,\ar@{--}"b";"al"
,\ar@{--}"al";"cl"
,\ar@{--}"ar";"cr"
,\ar@{--}"cl";"t"
,\ar@{--}"cr";"t"
,\ar@{--}"al";"cr"
,\ar@{--}"ar";"cl"
\end{xy}\ar@{-}[ddll]\ar@{-}@/_10ex/[dddd]\ar@{-}[ddllll]&
\begin{xy}
0;<2.75mm,0mm>:
,0*{\scalebox{1.5}{$\circ$}}="b"*++++!U{}
,(-3,2)*{\scalebox{1.5}{$\circ$}}="al"*++++!R{}
,(3,2)*{\scalebox{1.5}{$\circ$}}="ar"*++++!L{}
,(-3,5)*{\scalebox{1.5}{$\circ$}}="cl"*++++!R{}
,(3,5)*{\ci}="cr"*++++!L{}
,(0,7)*{\scalebox{1.5}{$\circ$}}="t"*++++!D{}
,\ar@{--}"b";"ar"
,\ar@{--}"b";"al"
,\ar@{--}"al";"cl"
,\ar@{--}"ar";"cr"
,\ar@{--}"cl";"t"
,\ar@{--}"cr";"t"
,\ar@{--}"al";"cr"
,\ar@{--}"ar";"cl"
\end{xy}
\ar@{-}[ddl]\ar@{-}[ddddr]\ar@{-}[ddr]&
\begin{xy}
0;<2.75mm,0mm>:
,0*{\scalebox{1.5}{$\circ$}}="b"*++++!U{}
,(-3,2)*{\scalebox{1.5}{$\circ$}}="al"*++++!R{}
,(3,2)*{\scalebox{1.5}{$\circ$}}="ar"*++++!L{}
,(-3,5)*{\scalebox{1.5}{$\circ$}}="cl"*++++!R{}
,(3,5)*{\scalebox{1.5}{$\circ$}}="cr"*++++!L{}
,(0,7)*{\ci}="t"*++++!D{}
,\ar@{--}"b";"ar"
,\ar@{--}"b";"al"
,\ar@{--}"al";"cl"
,\ar@{--}"ar";"cr"
,\ar@{--}"cl";"t"
,\ar@{--}"cr";"t"
,\ar@{--}"al";"cr"
,\ar@{--}"ar";"cl"
,\ar@{-}(-4.5,-1.5);(4.5,-1.5)
,\ar@{-}(4.5,-1.5);(4.5,8.5)
,\ar@{-}(4.5,8.5);(-4.5,8.5)
,\ar@{-}(-4.5,-1.5);(-4.5,8.5)
\end{xy}\ar@{-}[ddddll]\ar@{-}@/_10ex/[dddd]\\
{}&{}&{}&{}&{}&{}&{}\\
\begin{xy}
0;<2.75mm,0mm>:
,0*{\scalebox{1.5}{$\circ$}}="b"*++++!U{}
,(-3,2)*{\ci}="al"*++++!R{}
,(3,2)*{\scalebox{1.5}{$\circ$}}="ar"*++++!L{}
,(-3,5)*{\ci}="cl"*++++!R{}
,(3,5)*{\scalebox{1.5}{$\circ$}}="cr"*++++!L{}
,(0,7)*{\scalebox{1.5}{$\circ$}}="t"*++++!D{}
,\ar@{--}"b";"ar"
,\ar@{--}"b";"al"
,\ar@{-}"al";"cl"
,\ar@{--}"ar";"cr"
,\ar@{--}"cl";"t"
,\ar@{--}"cr";"t"
,\ar@{--}"al";"cr"
,\ar@{--}"ar";"cl"
\end{xy}\ar@{-}@/_10ex/[dddd]\ar@{-}@/^2ex/[ddddrrrrrr]&{}&
\begin{xy}
0;<2.75mm,0mm>:
,0*{\scalebox{1.5}{$\circ$}}="b"*++++!U{}
,(-3,2)*{\scalebox{1.5}{$\circ$}}="al"*++++!R{}
,(3,2)*{\ci}="ar"*++++!L{}
,(-3,5)*{\ci}="cl"*++++!R{}
,(3,5)*{\scalebox{1.5}{$\circ$}}="cr"*++++!L{}
,(0,7)*{\scalebox{1.5}{$\circ$}}="t"*++++!D{}
,\ar@{--}"b";"ar"
,\ar@{--}"b";"al"
,\ar@{--}"al";"cl"
,\ar@{--}"ar";"cr"
,\ar@{--}"cl";"t"
,\ar@{--}"cr";"t"
,\ar@{--}"al";"cr"
,\ar@{-}"ar";"cl"
\end{xy}\ar@{-}[ddddll]\ar@{-}[ddddrr]&{}&
\begin{xy}
0;<2.75mm,0mm>:
,0*{\scalebox{1.5}{$\circ$}}="b"*++++!U{}
,(-3,2)*{\ci}="al"*++++!R{}
,(3,2)*{\scalebox{1.5}{$\circ$}}="ar"*++++!L{}
,(-3,5)*{\scalebox{1.5}{$\circ$}}="cl"*++++!R{}
,(3,5)*{\ci}="cr"*++++!L{}
,(0,7)*{\scalebox{1.5}{$\circ$}}="t"*++++!D{}
,\ar@{--}"b";"ar"
,\ar@{--}"b";"al"
,\ar@{--}"al";"cl"
,\ar@{--}"ar";"cr"
,\ar@{--}"cl";"t"
,\ar@{--}"cr";"t"
,\ar@{-}"al";"cr"
,\ar@{--}"ar";"cl"
\end{xy}\ar@{-}[ddddrr]\ar@{-}@/_10ex/[dddd]&{}&
\begin{xy}
0;<2.75mm,0mm>:
,0*{\scalebox{1.5}{$\circ$}}="b"*++++!U{}
,(-3,2)*{\scalebox{1.5}{$\circ$}}="al"*++++!R{}
,(3,2)*{\ci}="ar"*++++!L{}
,(-3,5)*{\scalebox{1.5}{$\circ$}}="cl"*++++!R{}
,(3,5)*{\ci}="cr"*++++!L{}
,(0,7)*{\scalebox{1.5}{$\circ$}}="t"*++++!D{}
,\ar@{--}"b";"ar"
,\ar@{--}"b";"al"
,\ar@{--}"al";"cl"
,\ar@{-}"ar";"cr"
,\ar@{--}"cl";"t"
,\ar@{--}"cr";"t"
,\ar@{--}"al";"cr"
,\ar@{--}"ar";"cl"
\end{xy}\ar@{-}@/_15ex/[ddddllll]\ar@{-}[ddddll]\\
{}&{}&{}&{}&{}&{}&{}\\
\begin{xy}
0;<2.75mm,0mm>:
,0*{\ci}="b"*++++!U{}
,(-3,2)*{\ci}="al"*++++!R{}
,(3,2)*{\scalebox{1.5}{$\circ$}}="ar"*++++!L{}
,(-3,5)*{\scalebox{1.5}{$\circ$}}="cl"*++++!R{}
,(3,5)*{\scalebox{1.5}{$\circ$}}="cr"*++++!L{}
,(0,7)*{\scalebox{1.5}{$\circ$}}="t"*++++!D{}
,\ar@{--}"b";"ar"
,\ar@{-}"b";"al"
,\ar@{--}"al";"cl"
,\ar@{--}"ar";"cr"
,\ar@{--}"cl";"t"
,\ar@{--}"cr";"t"
,\ar@{--}"al";"cr"
,\ar@{--}"ar";"cl"
\end{xy}
\ar@{-}[dd]\ar@{-}[ddrr]&{}&\begin{xy}
0;<2.75mm,0mm>:
,0*{\ci}="b"*++++!U{}
,(-3,2)*{\scalebox{1.5}{$\circ$}}="al"*++++!R{}
,(3,2)*{\ci}="ar"*++++!L{}
,(-3,5)*{\scalebox{1.5}{$\circ$}}="cl"*++++!R{}
,(3,5)*{\scalebox{1.5}{$\circ$}}="cr"*++++!L{}
,(0,7)*{\scalebox{1.5}{$\circ$}}="t"*++++!D{}
,\ar@{-}"b";"ar"
,\ar@{--}"b";"al"
,\ar@{--}"al";"cl"
,\ar@{--}"ar";"cr"
,\ar@{--}"cl";"t"
,\ar@{--}"cr";"t"
,\ar@{--}"al";"cr"
,\ar@{--}"ar";"cl"
\end{xy}\ar@{-}[dd]\ar@{-}[ddll]&{}&\begin{xy}
0;<2.75mm,0mm>:
,0*{\scalebox{1.5}{$\circ$}}="b"*++++!U{}
,(-3,2)*{\scalebox{1.5}{$\circ$}}="al"*++++!R{}
,(3,2)*{\scalebox{1.5}{$\circ$}}="ar"*++++!L{}
,(-3,5)*{\ci}="cl"*++++!R{}
,(3,5)*{\scalebox{1.5}{$\circ$}}="cr"*++++!L{}
,(0,7)*{\ci}="t"*++++!D{}
,\ar@{--}"b";"ar"
,\ar@{--}"b";"al"
,\ar@{--}"al";"cl"
,\ar@{--}"ar";"cr"
,\ar@{-}"cl";"t"
,\ar@{--}"cr";"t"
,\ar@{--}"al";"cr"
,\ar@{--}"ar";"cl"
,\ar@{-}(-4.5,-1.5);(4.5,-1.5)
,\ar@{-}(4.5,-1.5);(4.5,8.5)
,\ar@{-}(4.5,8.5);(-4.5,8.5)
,\ar@{-}(-4.5,-1.5);(-4.5,8.5)
\end{xy}\ar@{-}[dd]\ar@{-}[ddrr]&{}&
\begin{xy}
0;<2.75mm,0mm>:
,0*{\scalebox{1.5}{$\circ$}}="b"*++++!U{}
,(-3,2)*{\scalebox{1.5}{$\circ$}}="al"*++++!R{}
,(3,2)*{\scalebox{1.5}{$\circ$}}="ar"*++++!L{}
,(-3,5)*{\scalebox{1.5}{$\circ$}}="cl"*++++!R{}
,(3,5)*{\ci}="cr"*++++!L{}
,(0,7)*{\ci}="t"*++++!D{}
,\ar@{--}"b";"ar"
,\ar@{--}"b";"al"
,\ar@{--}"al";"cl"
,\ar@{--}"ar";"cr"
,\ar@{--}"cl";"t"
,\ar@{-}"cr";"t"
,\ar@{--}"al";"cr"
,\ar@{--}"ar";"cl"
,\ar@{-}(-4.5,-1.5);(4.5,-1.5)
,\ar@{-}(4.5,-1.5);(4.5,8.5)
,\ar@{-}(4.5,8.5);(-4.5,8.5)
,\ar@{-}(-4.5,-1.5);(-4.5,8.5)
\end{xy}\ar@{-}[ddll]\ar@{-}[dd]\\
{}&{}&{}&{}&{}&{}&{}\\
\begin{xy}
0;<2.75mm,0mm>:
,0*{\ci}="b"*++++!U{}
,(-3,2)*{\ci}="al"*++++!R{}
,(3,2)*{\ci}="ar"*++++!L{}
,(-3,5)*{\ci}="cl"*++++!R{}
,(3,5)*{\scalebox{1.5}{$\circ$}}="cr"*++++!L{}
,(0,7)*{\scalebox{1.5}{$\circ$}}="t"*++++!D{}
,\ar@{-}"b";"ar"
,\ar@{-}"b";"al"
,\ar@{-}"al";"cl"
,\ar@{--}"ar";"cr"
,\ar@{--}"cl";"t"
,\ar@{--}"cr";"t"
,\ar@{-}"al";"cr"
,\ar@{-}"ar";"cl"
\end{xy}\ar@{-}[ddrrr]&{}&\begin{xy}
0;<2.75mm,0mm>:
,0*{\ci}="b"*++++!U{}
,(-3,2)*{\ci}="al"*++++!R{}
,(3,2)*{\ci}="ar"*++++!L{}
,(-3,5)*{\scalebox{1.5}{$\circ$}}="cl"*++++!R{}
,(3,5)*{\ci}="cr"*++++!L{}
,(0,7)*{\scalebox{1.5}{$\circ$}}="t"*++++!D{}
,\ar@{-}"b";"ar"
,\ar@{-}"b";"al"
,\ar@{-}"al";"cl"
,\ar@{-}"ar";"cr"
,\ar@{--}"cl";"t"
,\ar@{--}"cr";"t"
,\ar@{-}"al";"cr"
,\ar@{--}"ar";"cl"
\end{xy}\ar@{-}[ddr]&{}&
\begin{xy}
0;<2.75mm,0mm>:
,0*{\scalebox{1.5}{$\circ$}}="b"*++++!U{}
,(-3,2)*{\scalebox{1.5}{$\circ$}}="al"*++++!R{}
,(3,2)*{\ci}="ar"*++++!L{}
,(-3,5)*{\ci}="cl"*++++!R{}
,(3,5)*{\ci}="cr"*++++!L{}
,(0,7)*{\ci}="t"*++++!D{}
,\ar@{--}"b";"ar"
,\ar@{--}"b";"al"
,\ar@{--}"al";"cl"
,\ar@{-}"ar";"cr"
,\ar@{-}"cl";"t"
,\ar@{-}"cr";"t"
,\ar@{--}"al";"cr"
,\ar@{-}"ar";"cl"
,\ar@{-}(-4.5,-1.5);(4.5,-1.5)
,\ar@{-}(4.5,-1.5);(4.5,8.5)
,\ar@{-}(4.5,8.5);(-4.5,8.5)
,\ar@{-}(-4.5,-1.5);(-4.5,8.5)
\end{xy}\ar@{-}[ddl]&{}&
{\begin{xy}
0;<2.75mm,0mm>:
,0*{\scalebox{1.5}{$\circ$}}="b"*++++!U{}
,(-3,2)*{\ci}="al"*++++!R{}
,(3,2)*{\scalebox{1.5}{$\circ$}}="ar"*++++!L{}
,(-3,5)*{\ci}="cl"*++++!R{}
,(3,5)*{\ci}="cr"*++++!L{}
,(0,7)*{\ci}="t"*++++!D{}
,\ar@{--}"b";"ar"
,\ar@{--}"b";"al"
,\ar@{-}"al";"cl"
,\ar@{--}"ar";"cr"
,\ar@{-}"cl";"t"
,\ar@{-}"cr";"t"
,\ar@{-}"al";"cr"
,\ar@{--}"ar";"cl"
,\ar@{-}(-4.5,-1.5);(4.5,-1.5)
,\ar@{-}(4.5,-1.5);(4.5,8.5)
,\ar@{-}(4.5,8.5);(-4.5,8.5)
,\ar@{-}(-4.5,-1.5);(-4.5,8.5)
\end{xy}}
\ar@{-}[ddlll]\\
{}&{}&{}&{}&{}&{}&{}\\
{}&{}&{}&
{\begin{xy}
0;<2.75mm,0mm>:
,0*{\ci}="b"*++++!U{}
,(-3,2)*{\ci}="al"*++++!R{}
,(3,2)*{\ci}="ar"*++++!L{}
,(-3,5)*{\ci}="cl"*++++!R{}
,(3,5)*{\ci}="cr"*++++!L{}
,(0,7)*{\ci}="t"*++++!D{}
,\ar@{-}"b";"ar"
,\ar@{-}"b";"al"
,\ar@{-}"al";"cl"
,\ar@{-}"ar";"cr"
,\ar@{-}"cl";"t"
,\ar@{-}"cr";"t"
,\ar@{-}"al";"cr"
,\ar@{-}"ar";"cl"
,\ar@{-}(-4.5,-1.5);(4.5,-1.5)
,\ar@{-}(4.5,-1.5);(4.5,8.5)
,\ar@{-}(4.5,8.5);(-4.5,8.5)
,\ar@{-}(-4.5,-1.5);(-4.5,8.5)
\end{xy}}&{}&{}&{}\\
}
}
\end{center}
\end{figure}
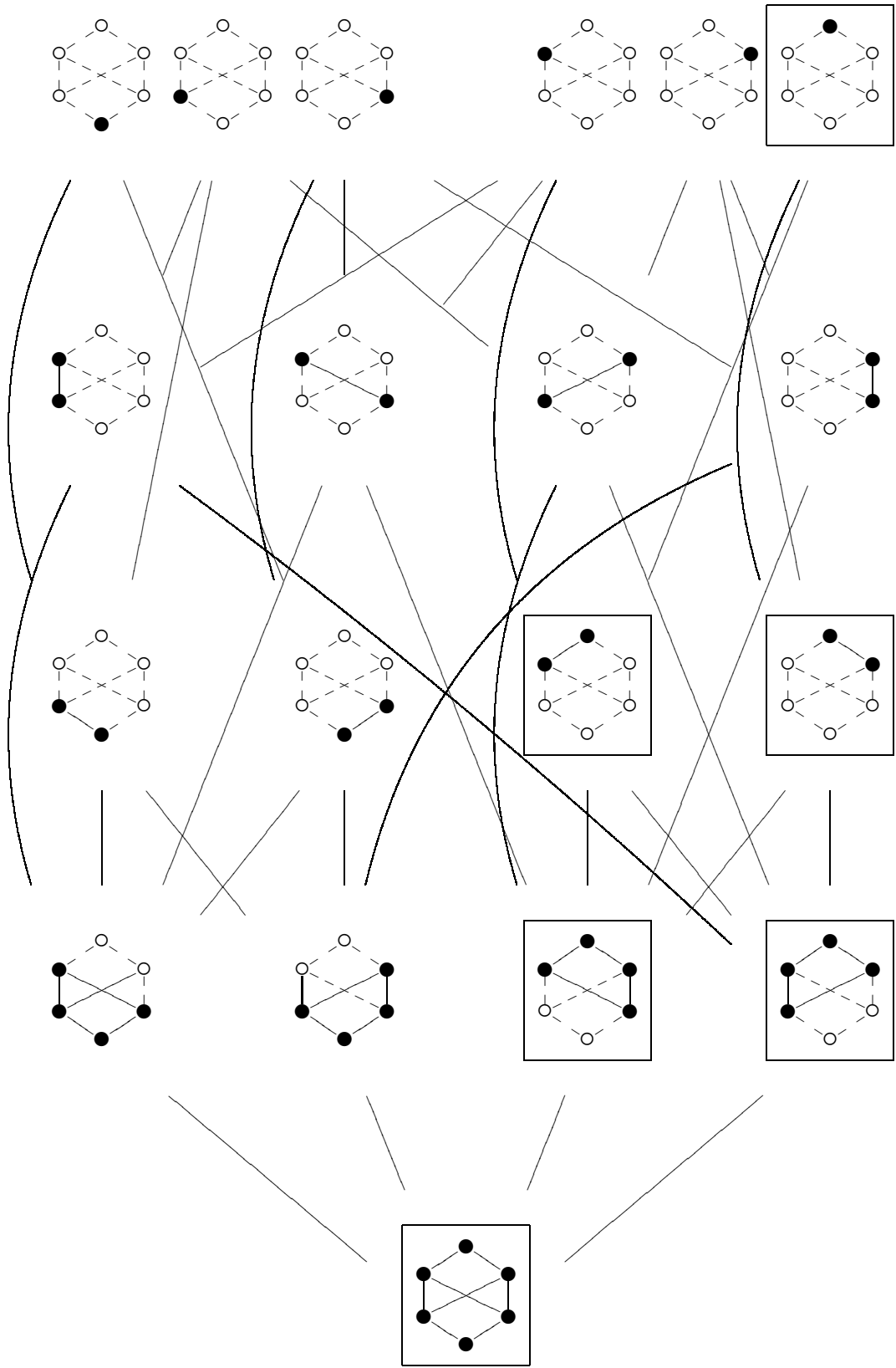

\subsection{variants of descent numbers}

Recall that $\dim (I, X, J)=|S\sm I|+|S\sm J|-1$ the dimension function of 
$\Xi$. This is essentially counting a part of the ascent (descent) number of $\min x$ ($\maxx$).
How can we consider some analogue of this for $\D$?
We will use the number of \eh{weak coatoms} of a coset.
For this purpose, let us prepare several definitions.

{\renewcommand{\arraystretch}{1.75}
\begin{table}[h!]
\caption{variants of descent numbers on $A_{2}$}
\label{t2}
\begin{center}
	\begin{tabular}{|c|ccccc|cccccccc}\h
$w$ &$d_{L1}(w)$&$d_{R1}(w)$&$d_{2LR}(w)$&$d(w)$&$\wt{d}(w)$\\\h
123	&0&0&0&0&	0\\\h
132	&0&0&1&1&	2\\\h
213	&0&0&1&1&	2\\\h
231&1&1&0&2&2\\\h
312&1&1&0&2&2\\\h
321&0&0&2&2&4\\\h
\end{tabular}
\end{center}
\end{table}
}

\begin{table}
{\renewcommand{\arraystretch}{1.75}
\caption{two-sided and total descent numbers $d(w), \wt{d}(w)$ over $A_3$}
\label{t3}
\begin{center}
\begin{tabular}{|c|c|c|c|c|c|c|c|c|c|c|c|c|c|c|c|}\hline
$w$&$d(w)$&$\wt{d}(w)$&$w$&$d(w)$&$\wt{d}(w)$&$w$&$d(w)$&$\wt{d}(w)$&$w$&$d(w)$&$\wt{d}(w)$\\\hline
1234&0&0&2134&$1$&2&3124&2&2&4123&2&2\\\hline
1243&1&2&2143&2&4&3142&3&3&4132&3&4\\\hline
1324&1&2&2314&2&2&3214&2&4&4213&3&4\\\hline
1342&2&2&2341&2&2&3241&3&4&4231&4&4\\\hline
1423&2&2&2413&3&3&3412&2&2&4312&3&4\\\hline
1432&2&4&2431&3&4&3421&3&4&4321&3&6\\\hline
\end{tabular}
\end{center}
\label{default}}
\end{table}%


\begin{defn}
A left descent $s\indlw$ is \eh{small} if $w^{-1}sw\not\in S$. Otherwise it is \eh{large}.
We use similar terminology for right descents.
\end{defn}
More notation:
\[D_{L1}(w)=\{s\in \dlw\mid w^{-1}sw\not\in S\},\]
\[D_{L2}(w)=\{s\in \dlw \mid w^{-1}sw\in S\},\]
\[D_{R1}(w)=\{s\in \drw\mid wsw^{-1}\not\in S\},\]
\[D_{R2}(w)=\{s\in \drw \mid wsw^{-1}\in S\},\]
\[
D_{L}(w)=D_{L1}(w)\cup D_{L2}(w), \q 
D_{R}(w)=D_{R1}(w)\cup D_{R2}(w), 
\q \te{(disjoint)}
\]
\[
D_{2LR}(w)=\{(r, s)\in D_{L2}(w)\ti D_{R2}(w)\mid rw=ws\}.
\]
Observe that $|D_{L2}(w)|=|D_{R2}(w)|=|D_{2LR}(w)|$.

\eh{Left small descent number, left large descent number}:
\[
d_{L1}(w)=|D_{L1}(w)|,
d_{L2}(w)=|D_{L2}(w)|,
\]
\eh{Right small descent number, right large descent number}:
\[
d_{R1}(w)=|D_{R1}(w)|,
d_{R2}(w)=|D_{R2}(w)|.
\]
\eh{Two-sided descent number}:
\[
d(w)=d_{L1}(w)+d_{R1}(w)+d_{2LR}(w).
\]
\eh{Total descent number}:
\[
\wt{d}(w)=d_{L1}(w)+d_{R1}(w)+2d_{2LR}(w)
\,\,(=|\dlw|+|\drw|).
\]
Clearly, $d(w)\le \wt{d}(w)\le 2d(w)$. Table \ref{t2} shows some examples.

Our method is to investigate double cosets with its maximal representative some fixed element (opposite to Billey et.al and Petersen).
Let 
\[
\D(w)=\{X\in \D\mid \maxx=w\}
\]
and $\delta(w)=|\D(w)|.$ 
Call $\D(w)$ the \eh{$w$-component} of $\D$.
Naturally, $\D=\cup_{w\inw}\D(w)$ and the union is disjoint.
Note $\min \D(w)= W_{\dlw}w W_{\drw}$ and 
$\max \D(w)=\{w\}$.
Although Billey et. al found the enumeration  formula 
\cite[Theorem 1.2]{billey5} with the ``marine model", 
we simply give some upper bound on $\delta(w)$ here.

\renewcommand{\side}{\tn{side}}
\renewcommand{\md}{\tn{mid}}

\begin{prop}\label{2d}
\begin{align*}
	\d(w)&\le 2^{\wt{d}(w)}.
\end{align*}
\end{prop}

\begin{proof}
Every coset $X$ with $\maxx=w$ has a marked coset expression 
$(I, W_{I}wW_{J}, J)$ with some $I\sub \dlw, J\sub \drw$.
Notice that 
$\dlw= 
D_{L1}(w)\cup 
D_{L2}(w)$ and 
$\drw= 
D_{R1}(w)\cup 
D_{R2}(w)$ are disjoint sums as introduced above.
Thus these $I, J$ can be expressed uniquely as 
\[
I=I_{1}\cup I_{2}, \q I_{1}\sub D_{L1}(w), I_{2}\sub D_{L2}(w),
\]
\[
J=J_{1}\cup J_{2},\q J_{1}\sub D_{R1}(w), J_{2}\sub D_{R2}(w).
\]
Hence 
\[
\d(w)\le |2^{D_{L1}(w)}||2^{D_{L2}(w)}||2^{D_{R1}(w)}||2^{D_{R2}(w)}|
=2^{d_{L1}(w)}
2^{d_{L2}(w)}2^{d_{R1}(w)}2^{d_{R2}(w)}
=2^{\wt{d}(w)}.
\]
\end{proof}

For example, let $w=54312$. We see that 
\[
\dlw=\{s_{2}, s_{3}, s_{4}\}, 
\drw=\{s_{1}, s_{2}, s_{3}\},
\]
\[
D_{L1}(w)=\{s_{2}\}, 
D_{L2}(w)=\{s_{3}, s_{4}\}, 
D_{R2}(w)=\{s_{1}, s_{2}\}, 
D_{R1}(w)=\{s_{3}\},
\]
$\wt{d}(w)=d_{L1}(w)+d_{L2}(w)+d_{R2}(w)+d_{R1}(w)=1+2+2+1=6$.
Therefore, $\delta(w)\le 2^{6}=32.$
\begin{ex}
Figure \ref{f3} illustrates
\[
\del(123)=1,
\del(213)=\del(132)=2,
\del(231)=\del(312)=4,
\]
and in particular $\del(321)=6$ as boxed cosets show.
\end{ex}
\begin{ex}
$W=A_{3}$ has 167 cosets and 281 marked cosets [OEIS A260700, A120733]. 
Thanks to Table \ref{t3}, we can check that 
\[
167=
\sum_{w\in A_{3}} \d(w)
\le
\sum_{w\in A_{3}} 2^{\wt{d}(w)}
\le
1\cdot 2^{0}+
10\cdot 2^{2}+
2\cdot 2^{3}+
10\cdot 2^{4}+
1\cdot 2^{6}=249<281.
\]
\end{ex}


\subsection{local structure: dimension, fiber, boolean complex}

\begin{defn}
The set of weak coatoms for $w$ is 
\[
C(w)=\{v\inw\mid v\lhd_{LR} w\}.
\]
Similarly, the set of weak coatoms for a double coset $X=[x_{0}, x_{1}]_{LR}$ is 
\[
C(X)=\{v\inx\mid v\lhd_{LR} x_{1}\}.
\]
Let $d(X)=|C(X)|$ (and $d(w)=|C(w)|$ as defined before).
Moreover, let $\wt{d}(X)=|M_{L}(X)|+|M_{R}(X)|$.
Call $d(X)$ the \eh{two-sided descent number} of $X$ and $\wt{d}(X)$ the \eh{total descent number} of $X$. 
\end{defn}
\begin{ob}
For each $X\in \D(w)$, we have the following:
\begin{enumerate}
	\item $0\le d(X)\le n, 0\le \wt{d}(X)\le 2n$.
	\item $d(X)=0\iff \wt{d}(X)=0$.
	\item $d(X)=n\iff \wt{d}(X)=2n$.
	\item $d(X)\le\wt{d}(X)\le 2d(X)$.
	\item $d(X)\le d(w)$, $\wt{d}(X)\le \wt{d}(w)$.
	\item If $(I, X, J)$ is a presentation of $X$, then 
	$d(X)\le |I|+|J|\le \wt{d}(X)$.
\end{enumerate}
\end{ob}

\begin{defn}
The \eh{local dimension} of $X\in \D(w)$ is 
\[
\dim_{\D(w)}(X)=d(w)-d(X)-1.\]
\end{defn}

If $X, Y\in \D(w)$ and $X\le Y$, then $Y\sub X$ so that $C(Y)\sub C(X)$, $d(Y)\le d(X)$ and $\dim_{\D(w)}(X)\le\dim_{\D(w)}(Y)$.
Thus, $\dim_{\D(w)}$ is a \eh{weakly} increasing function on $(\D(w), \le)$.
At the extremal cases, we have 
$\dim_{\D(w)}( W_{\dlw}w W_{\drw})=d(w)-d(w)-1=-1$ 
is the minimum and 
$\dim_{\D(w)}(\{w\})=d(w)-0-1=d(w)-1$ is the maximum.
In particular, call this $d(w)-1$ the \eh{dimension} of $\D(w)$.
Moreover, for any $k$ with $-1\le k\le d(w)-1$, 
there exists some $X\in \D(w)$ such that  
$\dim_{\D(w)}X=k$ as easily shown.

\subsection{relation between $\Xi$ and $\D$}

Let $\Xi(w):=\{(I, X, J)\in \Xi\mid \maxx=w\}$.
This is a boolean subinterval in $\Xi$ of rank $\wt{d}(w)$ as seen from  the proof of Proposition \ref{2d} (again, Petersen proved essentially the same result on marked cosets with minimal representative fixed \cite[Theorem 9]{petersen1}); hence vertices $\Xi(w)$ and covering edges $\{(I, X, J)\lhd (I', X, J')\}$ form a connected subgraph in the Hasse diagram of $\Xi$.
For this reason, we call $\Xi(w)$ the \eh{$w$-component} of $\Xi(W)$. 
Clearly, $\Xi(W)$ is the disjoint union of these components:
\[
\Xi(W)=\bigcup_{w\in W}\Xi(w). \]
\begin{ob}
The natural projection $\pi:\Xi\to \D$ by $\p((I, X, J))=X$
is weakly order-preserving.
\[
(I, X, J)\le (I', X', J')\then 
\pi(I, X, J) \le \pi(I', X', J').
\]
Observe in particular that $\p(\Xi(w))=\D(w)$.
\end{ob}

\begin{defn}
As an analogy of 
$\dim_{\D(w)}X=d(w)-d(X)-1$, define the \eh{local dimension function} on $\Xi(w)$ by
\[
\dim_{\Xi(w)}\ixj=\wt{d}(w)-(|I|+|J|)-1.
\]
\end{defn}
By definition of marked cosets, $\dim_{\Xi(w)}$ is a \eh{strictly}  increasing function on $(\Xi(w), \le)$. At the extremal cases, 
\[
\dim_{\Xi(w)}(\dlw, W_{\dlw}w W_{\drw},\drw)=
\wt{d}(w)-(|\dlw|+|\drw|)-1=
-1\]
is the minimum and 
$\dim_{\Xi(w)}(\ku, \{w\},\ku)=\wt{d}(w)-1$ is the maximum. This is indeed the dimension function of $\Xi(w)$ as a boolean complex.




%

The following proposition describes 
the relation of two local dimension functions via the projection:
\begin{prop}
Let $(I, X, J)\lhd (I', X', J')$ in $\Xi(w)$
so that 
$\dim_{\Xi(w)}(I', X', J')-\dim_{\Xi(w)}(I, X, J)=1$.
Then, 
$\dim_{\D(w)}\pi(I', X', J')-
 \dim_{\D(w)}\pi(I, X, J)\in\{0, 1\}$.
Consequently, 
if $(I, X, J)\le (I', X', J')$ and 
$\dim_{\Xi(w)}(I', X', J')-\dim_{\Xi(w)}(I, X, J)=k$,
then 
$\dim_{\D(w)}\pi(I', X', J')- \dim_{\D(w)}\pi(I, X, J)\in\{0, 1, 2, \ds, k\}$.
\end{prop}

\begin{proof}
Let $(I, X, J)\lhd (I', X', J')$ in $\Xi(w)$.
By definition, this means 
$X=W_{I}wW_{J}, X'=W_{I'}wW_{J'}, \maxx=w=\max X'$
and either 
\begin{enumerate}
	\item $I=I'\cup\{s\}, J=J'$ for some $s\in\dlw, s\not\in I'$  or 
	\item $I=I', J=J'\cup\{s\}$ for some $s\in\drw, s\not\in J'$.
\end{enumerate}
Say, for the moment, (1) holds. Considering the labels of edges between $w$ and those coatoms, we have 
\begin{align*}
	C(X)&=\{rw\mid r\in I\}\cup \{wr\mid r\in J\},
	\\C(X')&=\{rw\mid r\in I'\}\cup \{wr\mid r\in J'\},
\end{align*}
and $C(X')\sub C(X)=C(X')\cup\{sw\}$.
Note that $sw$ may or may not be in $C(X)$. 
Hence, $d(X)-d(X')\in\{0, 1\}$, that is, 
\[
\dim_{\D(w)}\pi(I', X', J')- \dim_{\D(w)}\pi(I, X, J)\in\{0, 1\}.
\]
It is quite similar to show this in the case (2). 
The last part is shown by induction.
\end{proof}

We can say more on relation between $\Xi$ and $\D$ through the  projection $\p$. Recall that 
a finite poset $P$ is a \eh{boolean complex} (\eh{simplicial poset}) if 
\begin{enumerate}
	\item $\wh{0}\in P$ ($\wh{0}\le x$ for all $x\in P$),
	\item each lower interval $[\wh{0}, x]$ is a boolean poset.
\end{enumerate}

\begin{lem}
For each $X\in \D$, the fiber $\p^{-1}(X)$ is a boolean complex.
\end{lem}

\begin{proof}
Say $w=\maxx$ so that $X\in \D(w)$.
By definition, 
\[
\p^{-1}(X)=\{(I, X, J)\in \Xi\mid I, J\subs\}.
\]
Note that, in this poset, 
$\wh{0}=(M_{L}(X), X, M_{R}(X))$ is the unique minimal element.
Furthermore, for each $(I, X, J)\in \p^{-1}(X)$, the lower interval 
\[
[(M_{L}(X), X, M_{R}(X)), (I, X, J)]
\]
is boolean since 
this is a subinterval of $\Xi(w)$ which is a boolean poset, 
and in fact every subinterval of a boolean poset is also boolean.
\end{proof}

%

\begin{prop}
Let $X\lhd \{w\}$ in $\D$. Then $X$ is covered by exactly two elements (and one of them is $\{w\}$).
\end{prop}

\begin{proof}
Suppose $X\lhd \{w\}$ in $\D$. Let $(I, x, J)$ be the maximal presentation of $X$. 
Since $X$ cannot be a singleton set ($X\lhd \{w\}$), we must have 
$|I|+|J|\ge 1$.
Say $|I|\ge 1 $ and choose $s\in I$.
Then $X=W_{I}xW_{J}=W_{I}wW_{J} (w\in X)$
 so that $\{w\}\subsetneq \{w, sw\}\sub X$ $(w\ne sw)$.
Since $\{w\}$ covers $X$, the set $\{w, sw\}$ must be $X$. 
In fact, $X=\{w, sw\}=W_{\{s\}}w$ is certainly a coset 
with $|X|=2$. Hence $X$ is covered by exactly two elements 
$\{w\}$ and $\{sw\}$. For the case $|J|\ge 1$, a similar argument is possible.
\end{proof}

\subsection{example}

Let $W=A_{2}$, $S=\{s_{1}, s_{2}\}$ and $w=s_{1}s_{2}s_{1}$, the longest element. Figure \ref{f4} shows $\Xi(w)$ is a boolean interval of rank $4=\wt{d}(w)$ with $2^{\wt{d}(w)}=16$ marked cosets. 
It naturally splits into 6 boolean complexes of 
dimension $1, -1, 0, 0, -1, -1$ (from the bottom), respectively:
\[
|\Xi(321)|=2^{4}=16
=7+1+3+3+1+1.
\]
These correspond to $6=\delta(w)$ fibers for cosets in $\D(w)$.
Observe also that $\pi:\Xi(w)\to \D(w)$ maps a boolean interval of rank 4 to a dihedral interval of rank 2 (six boxed cosets in Figure \ref{f3}).

\subsection{global structure: adjacent components}

We discussed several ``local" properties of $\Xi$ and $\D$.
Here, we wish to present some ``global" structures of such systems.
First, we mention a less-known property on descent numbers. 
\begin{prop}
If $v\lhd_{L} w$, then $d_{R}(w)\in\{d_{R}(v), d_{R}(v)+1\}$.
\end{prop}
\begin{proof}
Suppose $v\lhd_{L} w$.
By property of the left weak order, we have 
\[
T_{R}(v)\subsetneq T_{R}(w)=T_{R}(v)\upl \{v^{-1}w\},
\]
\[
\underbrace{T_{R}(v)\cap S}_{D_{R}(v)}\subseteq 
 \underbrace{T_{R}(w)\cap S}_{D_{R}(w)}=
(T_{R}(v)\cap S)\upl (\{v^{-1}w\}\cap S).\]
It follows that 
$D_{R}(w)=D_{R}(v)\upl(\{v^{-1}w\}\cap S)$ and therefore 
$d_{R}(w)\in\{d_{R}(v), d_{R}(v)+1\}$.
\end{proof}
Now, consider the one-sided Coxeter complex $\Sigma=\Sigma(W)=\{xW_{I}\mid x\in W, I\subs\}$.
Let us call \[
\Sigma(w)=\{xW_{I}\mid x\in W, I\subs, \max xW_{I}=w\}
\]
the \eh{$w$-component} of $\Sigma$. 
This is a boolean poset of rank $|\drw|=d_{R}(w)$.
Say components $\Sigma(v)$ and $\Sigma(w)$ are (left) \emph{adjacent} if $v\lhd_{L}w$. Then, $d_{R}(w)\in \{d_{R}(v), d_{R}(v)+1\}$ as shown above. 
Roughly speaking, adjacent components have close dimension.

We can do similar discussions for other systems.
Say components $\Xi(v)$ and $\Xi(w)$ are (two-sided) \emph{adjacent} if $v\lhd_{2LR}w$. 
By symmetry of left and right weak orders, $v\lhd_{R}w$ implies 
$d_{L}(w)\in\{d_{L}(v), d_{L}(v)+1\}$.
Consequently, 
if $v\lhd_{2LR}w$, then $d_{L}(w)=d_{L}(v)+1, d_{R}(w)=d_{R}(v)+1$, 
$d_{L1}(w)=d_{L1}(v), d_{L2}(w)=d_{L2}(v)+1,$
$d_{R1}(w)=d_{R1}(v), d_{R2}(w)=d_{R2}(v)+1$
so that 
$\wt{d}(w)=\wt{d}(v)+2$.
In this way, adjacent components $\Xi(v)$, $\Xi(w)$ have close dimension: 
$\wt{d}(w)=\wt{d}(v)+2$.

Also, say components $\D(v)$ and $\D(w)$ are (two-sided) \eh{adjacent} if $v\lhd_{2LR}w$. Then, for the same reason, 
$d(w)=d(v)+1$.

It is not so obvious whether the the whole $\D$ is ranked or not as 
Peterson pointed out \cite[Remark 4]{petersen1}. We will study this point in future publication.

\subsection{theorem}

We summarize our results as a Theorem.

\begin{thm}\label{th1}
The double coset system $(\D(W), \le)$ has the following structures:
\begin{enumerate}
	\item Set-theoretically, $\D$ is a disjoint union of $\D(w)=\{X\mid \maxx=w\}$. In addition, each of them forms a connected subgraph of the Hasse diagram of $(\D, \le)$; 
$\D(w)$ has the local dimension function $\dimdw:\D(w)\to\{-1, 0, 1, \ds, d(w)-1\}$ which is weakly increasing and surjective. Moreover, $\D(w)$ has a unique minimal element 
$W_{\dlw}w W_{\drw}$ of local dimension $-1$. 
If $\D(v)$ and $\D(w)$ are adjacent, then $\dim \D(w)=\dim \D(v)+1$.
	\item Maximal elements of $\D$ are in bijection with elements of $W$.
	\item If a coset is covered by a maximal element in $\D$, 
	then it is covered by exactly two elements. (this is quite similar to 
	the property $\Xi$ is a pseudomanifold as Petersen proved)
	\item For each coset $X\in \D$, the fiber $\p^{-1}(X)$ is a boolean complex. Moreover, for each $w\in W$, 
\[
\Xi(w)=\bigcup_{X\in \D(w)}\p^{-1}(X) 
\]
gives a partition of a boolean interval of rank $\wt{d}(w)$ 
into $\d(w)$ boolean complexes.
\end{enumerate}
\end{thm}

It would be nice if we could apply some of these results (particularly (4)) to find another formula for $\delta(w)$.

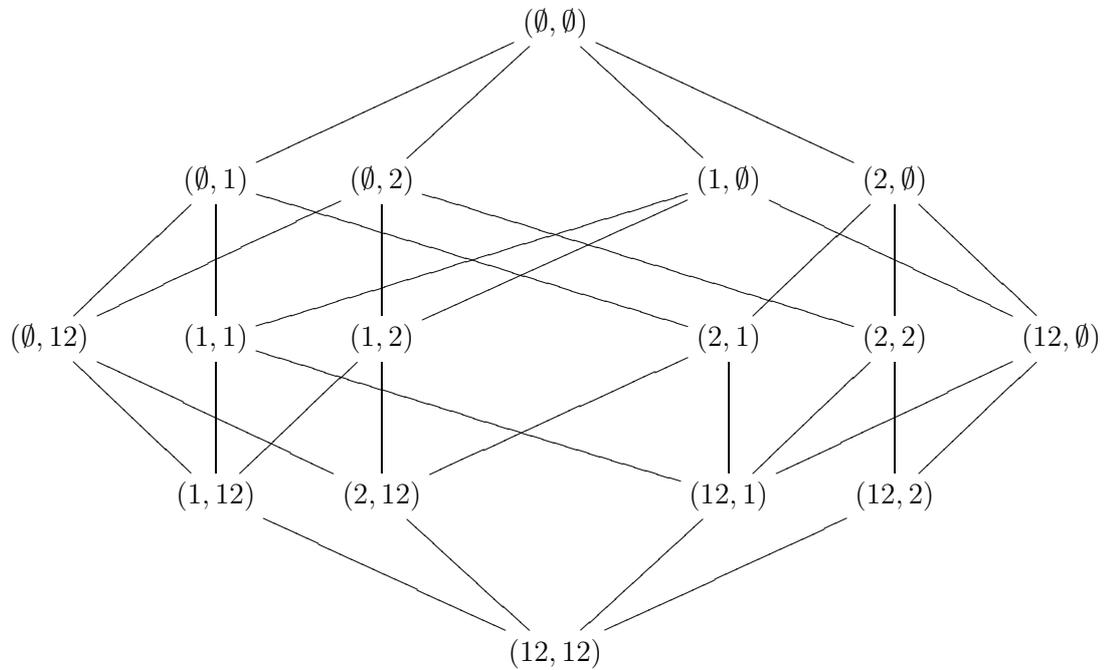
\begin{figure}
\caption{
boolean interval $\Xi(s_{1}s_{2}s_{1})$ of rank 4 
and its partition into 6 boolean complexes 
as 6 fibers of $\pi:\Xi(s_{1}s_{2}s_{1})\to \D(s_{1}s_{2}s_{1})$. 
($I, J$ abbreviated and indication of cosets omitted)}
\label{f4}
\begin{center}
\resizebox{\tw}{!}{
\xymatrix@R=16mm@C=10mm{
{}&{}&{}&(\ku, \ku)\ar@{-}[dll]\ar@{-}[dl]\ar@{-}[dr]\ar@{-}[drr]&{}&{}&{}\\
{}&(\ku, 1)\ar@{-}[dl]\ar@{-}[d]\ar@{-}[drrr]&(\ku, 2)\ar@{-}[dll]\ar@{-}[drrr]\ar@{-}[d]&{}&(1, \ku)\ar@{-}[dlll]\ar@{-}[dll]\ar@{-}[drr]&(2, \ku)\ar@{-}[dl]\ar@{-}[d]\ar@{-}[dr]&{}\\
(\ku, 12)\ar@{-}[dr]\ar@{-}[drr]&(1, 1)\ar@{-}[d]\ar@{-}[drrr]&(1, 2)\ar@{-}[dl]\ar@{-}[d]&{}&(2,1)\ar@{-}[dll]\ar@{-}[d]&(2,2)\ar@{-}[dl]\ar@{-}[d]&(12,\ku)\ar@{-}[dll]\ar@{-}[dl]\\
{}&(1,12)\ar@{-}[drr]&(2,12)\ar@{-}[dr]&{}&(12,1)\ar@{-}[dl]&(12,2)\ar@{-}[dll]&{}\\
{}&{}&{}&(12,12)&{}&{}&{}\\
}
}
\end{center}
\mb{}\\
\mb{}\\
\mb{}\\
\vfill\vf
\begin{center}
\resizebox{\tw}{!}{
\xymatrix@R=16mm@C=10mm{
{}&{}&{}&*+[F]{(\ku, \ku)}\ar@{}[dll]\ar@{}[dl]\ar@{}[dr]\ar@{}[drr]&{}&{}&{}\\
{}&(\ku, 1)\ar@{}[dl]\ar@{}[d]\ar@{-}[drrr]&(\ku, 2)\ar@{}[dll]\ar@{}[drrr]\ar@{-}[d]&{}&(1, \ku)\ar@{}[dlll]\ar@{-}[dll]\ar@{}[drr]&(2, \ku)\ar@{-}[dl]\ar@{}[d]\ar@{}[dr]&{}\\
(\ku, 12)\ar@{-}[dr]\ar@{-}[drr]&*+[F]{(1, 1)}\ar@{}[d]\ar@{}[drrr]&
*+[F]{(1, 2)}\ar@{}[dl]\ar@{}[d]&{}&*+[F]{(2,1)}\ar@{}[dll]\ar@{}[d]&*+[F]{(2,2)}\ar@{}[dl]\ar@{}[d]&(12,\ku)\ar@{-}[dll]\ar@{-}[dl]\\
{}&(1,12)\ar@{-}[drr]&(2,12)\ar@{-}[dr]&{}&(12,1)\ar@{-}[dl]&(12,2)\ar@{-}[dll]&{}\\
{}&{}&{}&*+[F]{(12,12)}&{}&{}&{}\\
}
}
\end{center}
\end{figure}

\section{Bruhat graphs on Bruhat intervals}

In this section, as applications of double cosets, we prove three theorems on degree of Bruhat graphs on Bruhat intervals.
It is helpful for understanding our discussion to keep the following idea in mind: for $u, v\in[e, w]$, define $u\simw v$
if 
\[
W_{\dlw}u W_{\drw}=W_{\dlw}v W_{\drw}.\] 
This gives a partition (an equivalent relation) of $[e, w]$:
\[
[e, w]=\bigcup_{u\le w}W_{\dlw}u W_{\drw}.
\]


\subsection{Carrell-Peterson's result}

\begin{defn}
The \eh{Poincar\'{e} polynomial} of $w$ is 
\[
\P_{w}(q)=
\sum_{v\le w}q^{\elv}.
\]
The \eh{average} of $\P_{w}(q)$ is 
$\P_{w}'(1)/\P_{w}(1)$.
\end{defn}

\begin{fact}[{\cite{bjorner2}}]
\label{kr}
There exists a unique family of polynomials $\{P_{uw}(q)\mid u, w \in W \} \subseteq \mathbb{Z}[q]$ (\emph{Kazhdan-Lusztig polynomials}) such that
\be{
\item $P_{uw}(q)=0$ if $u \not \le w$,
\item $P_{uw}(q)=1$ if $u=w$,
\item $\deg P_{uw}(q)\le (\ell(u, w)-1)/2$ if $u<w$,
\item \label{kr4}if $u\le w$, then
\begin{align*}
q^{\ell(u, w)}P_{uw}(q^{-1})=\sum_{u\le v\le w} R_{uv}(q)P_{vw}(q),
\end{align*}
where $\{R_{uv}(q)\mid u, v\inw\}$ are \eh{$R$-polynomials},
\item $[q^0](P_{uw})=1$ if $u\le w$.
}\ee
\end{fact}

\begin{fact}
Some invariance holds for a family of these polynomials:
If $r\in\dlw$ and $s\in\drw$, then 
$P_{ru,w}(q)=
P_{uw}(q)=
P_{us,w}(q).
$ 
Notice that this statement includes even the $u\not\le w$ case as 
$P_{ru,w}(q)=P_{uw}(q)=P_{us,w}(q)=0.$ 
\end{fact}

\begin{fact}[Carrell-Peterson \cite{carrell}]
\label{cp}
Suppose $P_{uw}(q)$ has nonnegative coefficients for all $u\le w$.
The following are equivalent:
\begin{quote}
	\begin{enumerate}
		\item $\P_{w}'(1)/\P_{w}(1)=\f{1}{2}\elw$.\st
		\item Every $v\in [u, w]$ is incident to $\eluw$ edges.
	\item $P_{uw}(q)=1$ for all $u\le w$.
	\item $q^{\elw}\P_{w}(q^{-1})=\P_{w}(q)$.
	\end{enumerate}
\end{quote}
\end{fact}

Note that Carrell-Peterson's assumption is now true for all $w$ in all Coxeter groups by Elias-Williamson \cite{elias-williamson}.

\subsection{regularity of Bruhat graphs}

The statement (2) in Fact \ref{cp} is about degree and regularity of graphs. Let us see more details on this idea.

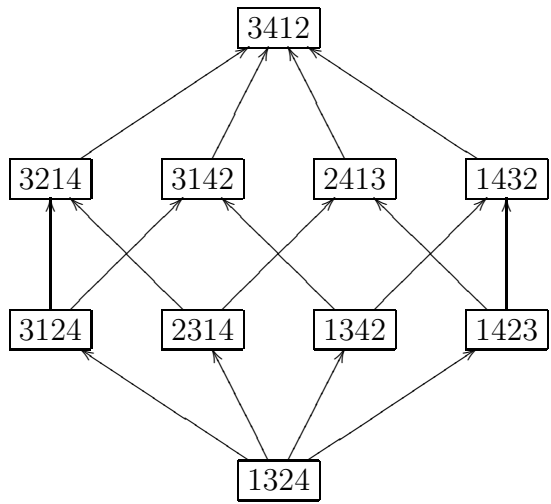
\begin{figure}[t]
\caption{an example of irregular Bruhat graphs\st}
\label{f5}
\mb{}\begin{center}
\begin{minipage}[t]{.7\tw}
\begin{xy}
(0,0);<20mm,0mm>:
,(0,1)*+[F]{3124}="101"
,(1,1)*+[F]{2314}="111"
,(2,1)*+[F]{1342}="121"
,(3,1)*+[F]{1423}="131"
,(0,2)*+[F]{3214}="102"
,(1,2)*+[F]{3142}="112"
,(2,2)*+[F]{2413}="122"
,(3,2)*+[F]{1432}="132"
,(1.5,0)*+[F]{1324}="1150"
,(1.5,3)*+[F]{3412}="1153"
,\ar@{->}"101";"102"
,\ar@{->}"1150";"101"
,\ar@{->}"101";"112"
,\ar@{->}"111";"102"
,\ar@{->}"102";"1153"
,\ar@{->}"112";"1153"
,\ar@{->}"121";"112"
,\ar@{->}"1150";"111"
,\ar@{->}"111";"122"
,\ar@{->}"1150";"121"
,\ar@{->}"1150";"131"
,\ar@{->}"121";"132"
,\ar@{->}"122";"1153"
,\ar@{->}"132";"1153"
,\ar@{->}"131";"122"
,\ar@{->}"131";"132"
\end{xy}
\end{minipage}
\end{center}
\end{figure}

\begin{defn}
Let $\deg_{V}(v)$ denote the degree of $v$ in Bruhat graph on the vertex set $V (\sub W)$.
Define the \eh{in-degree} and \eh{out-degree} of $v$:
\[
\inn_{V}(v)=\{u\in V\mid u\to v\},
\]
\[
\out_{V}(v)=\{u\in V\mid v\to u\}.
\]
For convenience, let $\inn_{V}(x)=\out_{V}(x)=0$ whenver $x\not\in V$.
\end{defn}
By definition, we have $\deg_{V}(v)=\inn_{V}(v)+\out_{V}(v).$
For $V=W$, we simply write 
$\deg_{W}(v)=\deg(v)$, 
$\inn_{W}(v)=\inn(v)\, (=\elv)$ and $\out_{W}(v)=\out(v)$. 
For $V=[e, w]$, we also simply write $\deg_{[e, w]}(v)=\deg_{w}(v)$
 and so on.
\begin{ques}
What can we say about degree of Bruhat graph on a lower interval?
\end{ques}

One remarkable fact is due to Deodhar \cite{deodhar}, Dyer \cite{dyer}  and Polo \cite{polo}:
\begin{fact}[Deodhar inequality \cite{dyer}]
We have \[
\deg_{w}(v)\ge \ell(w)\]
for all $v\in[e, w]$.
\end{fact}
(Similar statement holds for general Bruhat intervals)

Now recall that a finite directed graph is $d$-\eh{regular} if $\deg(v)=d$ for all vertex $v\in V$.
It is \eh{regular} if it is $d$-regular for some nonnegative integer $d$.
For example, the Hasse diagram of a finite boolean poset is regular;
the Bruhat graph of the whole $W=[e, w_{0}]$ is $|T|$-regular.

%
%

The following fact shows some special characteristic of Bruhat
 graphs.
\begin{fact}
The following are equivalent:
\begin{enumerate}
	\item $[e, w]$ is regular.
	\item $[e, w]$ is $\el(w)$-regular.
\end{enumerate}
\end{fact}
Remark: at a glance, it may be possible that 
even if there exists some $v\in [e, w]$ such that $\deg(v)>\elw$, the Bruhat graph $[e, w]$ still can be regular (i.e., $\deg(v)$-regular), but that is not true; notice that degree of the top element $w$ is exactly $\elw$.

One-sided cosets $W_{I}x, xW_{J}$ are always regular ($\el(w_{0}(I))$-regular, $\el(w_{0}(J))$-regular where 
$w_{0}(I), w_{0}(J)$ are the longest elements of $W_{I}, W_{J}$).
How about as short graphs? 
Indeed, $W_{I}x$, $xW_{J}$ are \eh{short-regular} meaning its short Bruhat graph is regular ($|I|$-regular, $|J|$-regular to be precise).
However, this does not hold for double cosets.
A counterexample appears even in $W=A_{2}$ ($s_{1}=(12), s_{2}=(23)$):
\[
X=W=W_{\{s_{1}, s_{2}\}}eW_{\{s_{1}, s_{2}\}}
\]
is itself a coset.
The short degree of $e=123$ in $X$ is $2$ while 
the one of $v=s_{1}=213$ is $3$.

{\renewcommand{\arraystretch}{1.5}
\begin{table}[h!]
\caption{regularity of finite cosets}
\label{t4}
\begin{center}
	\begin{tabular}{|c|c|c|c|cccc}\h
	&$W_{I}x$&$xW_{J}$	&$W_{I}xW_{J}$	\\\h
	as a short graph&regular&regular	&not necessarily regular\\\h
	as a Bruhat graph&regular&regular	&?\\\h
\end{tabular}
\end{center}
\end{table}}

\begin{ques}[Table \ref{t4}]
Is every double coset regular as a Bruhat graph?
\end{ques}
The answer is indeed yes as shown below;
this result might be proved in geometric method (such as theory of  Schubert varieties \cite{billey1} and Richardson varieties \cite{billey-coskun}), 
here we give a combinatorial proof.
\begin{thm}\label{th2}
Every double coset $X$ is regular. To be more precise, $X$ is $\el(X)$-regular.
\end{thm}

\begin{proof}
Let $X=[x_{0}, x_{1}]$ be a double coset.
For $x\inx$, define
\begin{align*}
	\tn{in}(x)&=|\{w\inw\mid w\to x, w\not\in X\}|,
	\\\self_{\up}(x)&=|\{w\inw\mid w\to x, w\in X\}|,
	\\\self^{\up}(x)&=|\{w\inw\mid x\to w, w\in X\}|.
\end{align*}
Observe that if $x=u*x_{0}*v$ with $X=W_{I}xW_{J}$,
$u\in W_{I}, v\in W_{J}$,
then 
\[
\inn(x)=\el(x_{0}), \q 
\self_{\up}(x)=\elu+\elv.
\]
(to see this, consider a reduced factorization	 
\[
x=a_{1}\cd a_{l}*b_{1}\cd b_{m}* c_{1}\cd c_{n}
\]
 for $x$ with $a_{i}, b_{j}, c_{k}\in S$,
 $u=a_{1}\cd a_{l}, x_{0}=b_{1}\cd b_{m}, v=c_{1}\cd c_{n}$ all reduced,  
$t_{i}=a_{1}\cd a_{i-1}a_{i}a_{i-1}\cd a_{1}\in T_{L}(u) \sub T_{L}(x)$
 and 
 $t'_{i}=c_{n}\cd c_{n-i+2}c_{n-i+1}c_{n-i+2}\cd c_{n}\in T_{R}(v) \sub T_{R}(x)$. It is easy to see $t_{i}x\ne xt'_{i}$ for all $i$.)
It follows that 
\[
\elx=\inn(x)+\self_{\up}(x),
\]
\[
\deg_{X}(x)=\self_{\up}(x)+\self^{\up}(x).
\]
Now we need to show that 
for all $x, y\inx$, $\deg_{X}(x)=\deg_{X}(y)$.
Since $X$ is a graded poset,
It is enough to show this for $(x,y) \in X^{2}$ with $x\lhd_{LR} y$.
Say, for the moment, $x\lhd_{L}y, y=sx, s\in I\sub A_{L}(x_{0})\cap D_{L}(x_{1})$.
Note that 
\[
\inn(x)=\el(x_{0})=\inn(y) \te{\mb{} and } 
\ely=\elx+1.
\]
We claim that 
for each $t\in T\sm\{x^{-1}sx\}$, we have 
\[
x\to xt\inx\iff y\to yt\inx.
\]
If $x\to xt\inx$ (in particular $t\not\in T_{L}(x)$), then $y\to yt$
since $T_{L}(x)\sub T_{L}(y)=T_{L}(x)\cup\{x^{-1}sx\}$. 
Thus $t\not\in T_{L}(y)$.
Moreover, \[
yt=(sx)t=\underbrace{s}_{\in 
A_{L}(x_{0})\cap D_{L}(x_{1})
}\underbrace{(xt)}_{\in X}\in X\]
by Lifting Property (and vice versa). We proved the claim.
If $x\lhd_{R}y, y=xs, s\in J\sub A_{R}(x_{0})\cap D_{R}(x_{1})$, 
then the similar proof goes with $t$ on the left side.
Consequently, 
\begin{align*}
	\deg_{X}(y)-\deg_{X}(x)&=
	\k{\self_{\up}(y)+\self^{\up}(y)}-
	\k{\self_{\up}(x)+\self^{\up}(x)}
	\\&=
	\k{\self_{\up}(y)-\self_{\up}(x)}
	+\k{\self^{\up}(y)-\self^{\up}(x)}
	\\&=(\ely-\inn(y))-(\elx-\inn(x))-1
	\\&=(\ely-\elx)-(\inn(y)-\inn(x))
	\\&=0.
\end{align*}
Hence $\deg(x)=\deg(x_{1})=\el(X)$.
\end{proof}

The following is an analogy of invariance of 
Kazhdan-Lusztig polynomials:
$P_{rv,w}(q)=
P_{vw}(q)=
P_{vs,w}(q)
$ 
for all $v\in W$, $r\in D_{L}(w)$ and $s\in\drw$.
\begin{thm}\label{th3} We have
\[
\deg_{w}(rv)=\deg_{w}(v)=\deg_{w}(vs)\]
for all $v\in W$, $r\in D_{L}(w)$ and $s\in\drw$.
\end{thm}

\begin{proof}
It is enough to show that $\deg_{w}(v)=\deg_{w}(rv)$
for $v\in [e, w]$ and $r\in D_{L}(w)$.
Replacing $v$ by $rv$ if necessary, we may assume that $v\to rv$.
Split $\deg_{w}(v)$ as:
\[
\deg_{w}(v)=
\inn_{w}(v)+\out_{w}(v).\]
It follows that 
\[
\inn_{w}(v)=\el(v), \q
\inn_{w}(rv)=\el(rv)=\el(v)+1. 
\]
Now we claim that for $t\in T\sm \{v^{-1}rv\}$, 
we have 
\[
v\to vt\le w\iff rv\to (rv)t \le w.\]
If $v\to vt\le w$ (in particular $t\not\in T_{R}(v)$), then $rv\to rvt$
since 
\[
T_{R}(v)\subsetneq T_{R}(rv)=T_{R}(v)\cup \{v^{-1}rv\}.\]
Moreover, $rvt=r\underbrace{(vt)}_{\le w}\le w$ by Lifting Property. 
We showed $(\then)$ (and vice versa). We proved the claim.
In addition, $v$ is incident to exactly one more outgoing edge $v\to rv\, (=v(v^{-1}rv)) \le w$. Hence
\[
\out_{w}(v)=\out_{w}(rv)+1. 
\]
Altogether, conclude that 
\[
\deg_{w}(v)=
\inn_{w}(v)+
\out_{w}(v)=
\elv+
\out_{w}(rv)+1=
\deg_{w}(rv). 
\]
\end{proof}


\subsection{out-Eulerian property}

Recall the basic fact that 
every Bruhat interval is Eulerian. In particular, 
\[
\sum_{v\in [e, w]}(-1)^{\el(v)}=0
\]
for $w\ne e$.
For a proof of this, choose $t\in T\cap W_{\dlw}$. 
We see that $v\leftrightarrow tv$ is a perfect 
matching on $[e, w]$ as a consequence of Lifting Property.
Moreover, as is well-known, $\el(v, tv)$ is always odd.
Hence 
\[
\sum_{v\in [e, w]}(-1)^{\el(v)}=
\sum_{v\leftrightarrow tv }\k{(-1)^{\el(v)}+(-1)^{\el(tv)}}
=0.
\]
(Usually, we take $t$ to be a simple reflection.
However, it is not necessary.)
In terms of Bruhat graphs, we can understand this Eulerian property as
\[
\sum_{v\le w}(-1)^{\inn_{w}(v)}=0.
\]
It is natural to wonder if the similar statement on out-degree holds.
The point is:if $u\to v$ in $[e, w]$, is always 
$\out_{w}(u)-\out_{w}(v)$ odd?
The answer is no; but this is far from obvious and not so often this idea has been mentioned in the literature.
\begin{defn}
Say an edge $u\to v$ in $[e, w]$ is \eh{out-odd}
if $\out_{w}(u)-\out_{w}(v)$ is odd.
It is \eh{out-even}
if $\out_{w}(u)-\out_{w}(v)$ is even.
\end{defn}
We can easily find an example of both kinds (Figure \ref{f5}): 
$3124\to 3214$ is out-odd 
in $[1324, 3412]$ 
since 
\[
\out_{3412}(3124)-\out_{3412}(3214)=2-1=1
\]
while 
$1324\to 3124$ is out-even since 
\[
\out_{3412}(1324)-\out_{3412}(3124)=4-2=2.
\]
We will show that ``out-Euerlian property" 
holds for some special class of Bruhat intervals:

\begin{defn}
We say that $(u, w)\in W\ti W$ is a \eh{critical pair} if 
\[
u\le w, \q D_L(w)\subseteq D_L(u) \,\text{   and   }\, D_R(w)\subseteq D_R(u).
\]
An interval $[u, w]$ is \eh{critical} if $(u, w)$ is a critical pair.
\end{defn}
Now we have a simple classification:
\begin{quote}
Bruhat intervals
$\begin{cases}
	\tn{critical}\\
	\tn{noncritical}
\end{cases}$
\end{quote}
A trivial interval is critical; a lower interval $[e, w]$ $(w\ne e)$ is noncritical; a double coset of length $\ge1$ is noncritical.
Observe that 
if $[e, w]$ is noncritical, then 
there exists some $s\in S$ such that 
$s\in A_{L}(u)\cap \dlw$ 
or 
$s\in A_{R}(u)\cap \drw$.


\begin{thm}[out-Eulerian Property]\label{th4}
For every noncritical interval $[u, w]$, we have
\[
\sum_{v\in [u, w]} (-1)^{\out_{[u, w]}(v)}=0.
\]
\end{thm}
Notice that $\out_{[u, w]}(v)=\out_{w}(v)$ for $v\in [u, w]$.
\begin{proof}
Consider the partition of $[e, w]$:
\[
[e, w]=
\bigcup_{v\in [e, w]} W_{\dlw}v W_{\drw}.
\]
Assume $[u ,w]$ is noncritical.
Then, (say left) $A_{L}(u)\cap \dlw\ne\ku$ and choose 
$t\in T\cap W_{\dlw}$.
For each $v\in[u, w]$, 
$v\leftrightarrow tv$ is a perfect matching on $[u, w]$
as a consequence of Lifting Property, again.
Moreover, $v\sim_{w}tv$ implies 
\[
\deg_{w}(v)=\deg_{w}(tv)
\]
as proved in Theorem \ref{th3}. 
It follows that 
\[
\inn_{w}(v)+\out_{w}(v)=\inn_{w}(tv)+\out_{w}(tv),
\]
\[
\elv+\out_{w}(v)=\el(tv)+\out_{w}(tv),
\]
\[
\elv-\el(tv)=\out_{w}(tv)-\out_{w}(v).
\]
Since $\elv-\el(tv)$ is odd, so is 
$\out_{w}(tv)-\out_{w}(v)$.
Conclude that 
\begin{align*}
	\sum_{v\in [u, w]}(-1)^{\out_{[u, w]}(v)}&=
	\sum_{v\in [u, w]}(-1)^{\out_{w}(v)}
	\\&=
	\sum_{v \leftrightarrow tv}\k{(-1)^{\out_{w}(v)}+(-1)^{\out_{w}(tv)}}
	\\&=0.
\end{align*}
\end{proof}

\section{Further remarks}

We end with recording some ideas for subsequent research.

\subsection{four-variable Eulerian polynomials}

It is possible to consider the following 
$W$-Eulerian polynomial in four variables
\[
A_{W}(t_{1}, t_{2}, t_{3}, t_{4})=
\sum_{w\inw} 
t_{1}^{d_{L1}(w)}
t_{2}^{d_{L2}(w)}
t_{3}^{d_{R2}(w)}
t_{4}^{d_{R1}(w)}.
\]
Notice that 
$t_{1}=t_{2}=t$ and $t_{3}=t_{4}=1$ recovers classical $W$-Eulerian polynomial (Brenti \cite{brenti}) and 
$t_{1}=t_{2}=s, t_{3}=t_{4}=t$ recovers two-sided $W$-Eulerian polynomial (Petersen \cite{petersen2}).
What if
 $t_{1}=t, t_{2}=t_{3}=t^{1/2}, t_{4}=t$ (
the generating function of \eh{two-sided descent number})
or $t_{1}=t_{2}=t_{3}=t_{4}=t$ (the generating function of \eh{total descent number})?

This polynomial must satisfy 
$|\D(W)|\le A_{W}(2, 2, 2, 2)$ 
and $A_{W}(t_{1}, t_{2}, t_{3}, t_{4})
=A_{W}(t_{4}, t_{3}, t_{2}, t_{1})$ 
since 
$d_{L1}(w^{-1})=d_{R1}(w), d_{L2}(w^{-1})=d_{R2}(w)$ and so on.

%

\subsection{left, right, central Poincar\'{e} polynomials}

As is well-known, Poincar\'{e} polynomials have nice factorization property with respect to  cosets and quotients \cite{bjorner2}.
Here let us consider more variants of Poincar\'{e} polynomials.
%
Recall that $w_{0}(I)$ means the longest element of $W_{I}$.

\begin{defn}
Define maps (\eh{left, coleft, right, coright projections})
 $L, \wt{L}, R, \wt{R}:W\to W$ by 
\[
\begin{matrix}
	L(x)=w_{0}(D_{L}(x)),&   & \wt{L}(x)=L(x)^{-1}x,  \\
	R(x)=w_{0}(D_{R}(x)),&   & \wt{R}(x)=xR(x)^{-1}.  
\end{matrix}
\]
%
%
so that we have
\[
x=L(x)\wt{L}(x) \text{ and } \el(x)=\el(L(x))+\el(\wt{L}(x))\]
\[
x=\wt{R}(x)R(x) \text{ and } \el(x)=\el(\wt{R}(x))+\el(R(x)).\]

Call $L(w)$ ($\wt{L}(w)$, $R(w)$, $\wt{R}(w)$) the \eh{left} (\eh{coleft, right, coright}) \eh{part} of $w$, 
and 
left length, left colength right length, right colength 

\[
\begin{matrix}
	\el_{L}(x)=\el(L(x)),&   & \el_{\wt{L}}(x)=\el(\wt{L}(x)),  \\
	\el_{R}(x)=\el(R(x)),&   & \el_{\wt{R}}(x)=\el(\wt{R}(x)).  
\end{matrix}
\]
\end{defn}

\begin{defn}
Let 
$C(w)=\min W_{\dlw} w W_{\drw}$ be the \eh{central projection}.
$\el_{C}(w)=\el(C(w))$,
$\el_{\side}(w)=\elw-\el_{C}(w)$:
\eh{central length} and \eh{side length} of $w$.
\end{defn}



For example,
$w=45312$ has a reduced word
$s_{2}s_{3}s_{2}s_{1}s_{4}s_{2}s_{3}s_{2}$
with $\dlw=\drw=\{s_{2}, s_{3}, s_{2}\}$.
Thus, 
\[
45312=\underbrace{s_{2}s_{3}s_{2}}_{L(w)}
\underbrace{s_{1}s_{4}s_{2}s_{3}s_{2}}_{\wt{L}(w)}
=
\underbrace{s_{2}s_{3}s_{2}s_{1}s_{4}}_{\wt{R}(w)}
\underbrace{s_{2}s_{3}s_{2}}_{R(w)}
=
s_{2}s_{3}s_{2}\underbrace{s_{1}s_{4}}_{C(w)}s_{2}s_{3}s_{2}
\]
\[
\el_{L}(w)=\el_{R}(w)=3,
\el_{\wt{L}}(w)=\el_{\wt{R}}(w)=5,
\el_{C}(w)=2,
\el_{\side}(w)=6.
\]
%
%
%
%
%
%

%

%

%
%
\eh{Left, Right, Central Poincare polynomials} of $w$:
\begin{align*}
	\P_{w}^{L}(q_{1}, q_{2})&=
	\sum_{v\le w}q_{1}^{\el_{L}(v)}q_{2}^{\el_{\wt{L}}(v)},
	\\\P_{w}^{R}(q_{1}, q_{2})&=
\sum_{v\le w}q_{1}^{\el_{\wt{R}}(v)}q_{2}^{\el_{R}(v)},
	\\\P_{w}^{C}(q_{1}, q_{2})&=
	\sum_{v\le w}q_{1}^{\el_{C}(v)}q_{2}^{\el_{\side}(v)}.
\end{align*}
In particular, 
$\P_{w}^{L}(q, q)=\P_{w}^{R}(q, q)=\P_{w}^{C}(q, q)=\P_{w}(q)$.
Find these polynomials.

\subsection{new enumeration problems on Bruhat graphs}

\begin{enumerate}
	\item 
Let $V(w)$ be the vertex set of $[e, w]$ and 
	$E(w)=\{ u\to v\mid u, v\in [e, w]\}$
the set of all edges.
	Say a vertex $v$ is \eh{irregular} if 
$\deg_{w}(v)>\elw$. 
Say an edge $u\to v$ is \eh{irregular} if it is 
incident to an irregular vertex; 
see \cite[Theorem 8.2]{kob5} some relation between irregularity and edges of Bruhat graphs.
The following rational numbers seem to be quite natural to ``measure irregurarity" of $[e, w]$:
\[
\ff{|V_{\tn{irr}}(w)|}{|V(w)|},
\ff{|E_{\tn{irr}}(w)|}{|E(w)|},
\]
However, these have not been studied. 
Compute some examples. Can these numbers be any rational number between 0 and 1?
	\item When is an edge $u\to v$ in $[e, w]$ out-even or when not?
Try Type A. Describe it in terms of reduced words, 
monotone triangles and pattern avoidance. 
	\item 
Further, the \eh{in-out-Poincar\'{e} polynomial} of $w$ is 
\[
\P_{w}^{\tn{in-out}}(q_{1}, q_{2})=\sum_{v\le w}q_{1}^{\inn_{w}(v)}q_{2}^{\tn{out}_{w}(v)}.
\]
In particular, \eh{out-Poincar\'{e} polynomial} of $w$ is 
\[
\P_{w}^{\tn{out}}(q)=\P_{w}^{\tn{in-out}}(1, q)=
\sum_{v\le w}q^{\tn{out}_{w}(v)}
\]
as we showed that $\P_{w}^{\tn{out}}(-1)=0$ for $w\ne e$.
Study these polynomials. When are they palindromic?
	\item 
	The following are equivalent \cite{billey1}:
\begin{enumerate}
	\item $[e, w]$ is irregular.
	\item There exists some $v\in [e, w]$ such that $\deg(v)>\elw$.
	\item $\deg_{w}(e)>\elw$.
\end{enumerate}
The degree function in (c) is interesting:
\[
\deg_{w}(e)=|\{v\in [e, w]\mid e\to v\}|
=|\{t\in T\mid t\le w \te{ (subword)}\}|.
\]
Let $\lam(w)=\deg_{w}(e)$.
By definition, $\lam$ is \eh{weakly} increasing in Bruhat order:
\[
x\le y\then \lam(x)\le \lam(y).
\]
Observe also that 
$\lam(e)=0=\el(e), \lam(w_{0})=|T|=\el(w_{0})$ 
and moreover $\lam(w)\ge \el(w)$ due to Deodhar inequality. 
Let us say $w$ is \eh{combinatorially smooth} \cite{abe-billey} if $\lam(w)=\elw$.
It is not so easy to predict 
when $\lam(w)>\el(w)$ as the example shows below:

\begin{center}
\resizebox{.9\tw}{!}{
\begin{tabular}{cccccccccccccccc}
$w$&	1234&$\to$   &1324 &$\to$ &3124&$\to$&3142&$\to$ &3412&$\to$ &4312&$\to$ &4321  \\
$\lam(w)$&	0&$\to$   &1&$\to$ &2&$\to$&3&$\to$ &5&$\to$ &5&$\to$ &6  \\
\end{tabular}
}
\end{center}
\begin{align*}
	\lam(3142)&=|\{t\in T\mid t\le s_{2}s_{1}s_{3}\}|
	=|\{s_{1}, s_{2}, s_{3}\}|=3,
	\\\lam(3412)&=|\{t\in T\mid t\le s_{2}s_{1}s_{3}s_{2}\}|
	\\&=|\{s_{1}, s_{2}, s_{3}, s_{2}s_{1}s_{2}, s_{2}s_{3}s_{2}\}|=5,
	\\\lam(4312)&=|\{t\in T\mid t\le s_{2}s_{1}s_{3}s_{2}s_{1}\}|
	\\&=|\{s_{1}, s_{2}, s_{3}, s_{2}s_{1}s_{2}, s_{2}s_{3}s_{2}\}|=5,
\end{align*}
Discuss edges $v\tow$ such that $\lam(v, w)=0, 1$ or $\ge 2$.
\end{enumerate}





%



\end{document}